\documentclass[12pt]{amsart}

\usepackage{accsupp}
\usepackage{amssymb,amscd,amsmath}
\usepackage{amsfonts}
\usepackage{fullpage}
\usepackage{array}
\usepackage{amsthm}
\usepackage{url}
\usepackage{mathrsfs}
\usepackage{todonotes}
\usepackage{relsize}
\usepackage{enumerate}
\usepackage{comment}

\numberwithin{equation}{section}
\newtheorem{theorem}{Theorem}[section]

\newtheorem{lemma}[theorem]{Lemma}

\theoremstyle{definition}

\newtheorem{remark}[theorem]{Remark}

\newtheorem*{remark*}{Remark}
\newtheorem*{proposition*}{Proposition}
\newtheorem*{acknowledgement*}{Acknowledgement}
\newtheorem{example}[theorem]{Example}
\newtheorem{definition}[theorem]{Definition}
\newtheorem{lem}[theorem]{Lemma}

\newtheorem{prop}[theorem]{Proposition}
\newtheorem{cor}[theorem]{Corollary}

\newcommand{\Z}{\mathbb{Z}}
\newcommand{\Q}{\mathbb{Q}}
\newcommand{\R}{\mathbb{R}}
\newcommand{\C}{\mathbb{C}}

\newcommand{\HH}{\mathcal{H}}

\newcommand{\hyp}[4]{~_2F_1\left(\left.\begin{smallmatrix}#1,~#2\\#3\end{smallmatrix}\right\rvert#4\right)}
\renewcommand{\mod}{\text{ mod }}
\renewcommand{\Im}{\text{Im\,}}
\renewcommand{\Re}{\text{Re\,}}

\DeclareMathOperator{\Res}{Res}

\title{Ratios of Artin $L$-functions}

\author[L. Hochfilzer]{Leonhard Hochfilzer}
\author[T. Oliver]{Thomas Oliver}

\address{(L.H.) Mathematisches Institut, Bunsenstrasse 3-5, 37073 G\"{o}ttingen, Germany.}
\address{(T.O.) School of Mathematical Sciences, The University of Nottingham, University Park, Nottingham, NG7 2RD.}

\email{leonhard.hochfilzer@mathematik.uni-goettingen.de, thomas.oliver@nottingham.ac.uk.}

\date{\today}

\usepackage{hyperref}
\hypersetup{pdfauthor={Leonhard Hochfilzer and Thomas Oliver},
	            pdftitle={Artin ratios},
	            pdfsubject={Number Theory},
	            pdfcreator={XeLaTeX},
	            colorlinks=true,
	            linkcolor=[rgb]{0,.6,1},
	            citecolor=blue}

\begin{document}


\maketitle

\subsection*{Abstract}
We study the cancellation of zeros between the Riemann zeta function and certain Artin $L$-functions.
To do so, we develop a converse theorem for Maass forms of Laplace eigenvalue $1/4$ in which the twisted $L$-functions are not assumed to be entire. 
We do not require the conjectural automorphy of Artin $L$-functions, only their established meromorphic continuation and functional equation.

\section{Introduction}\label{sec.intro}

This paper studies cancellation of zeros between the Riemann zeta function and the Artin $L$-functions associated with $3$-dimensional Artin representations. 
The motivations are two-fold. 
Firstly, our main Theorem offers modest evidence for the Grand Simplicity Hypothesis (which implies that unrelated pairs of $L$-functions cannot have common zeros in the critical strip).  
Secondly, as we shall explore, the proof of our main Theorem offers an interesting connection with the question of Langlands functoriality for Artin representations. 

A typical $L$-function is a Dirichlet series $L(s)$ converging in a right half-plane, admitting continuation to $\mathbb{C}$, and satisfying a standardised functional equation with respect to $s\mapsto1-s$. 
The functional equation for $L(s)$ is best expressed in terms of its completion $\Lambda(s)$, which has the form $A^sL(s)\prod_{i=1}^{r}\Gamma(\lambda_is+\mu_i)$ for a positive real number $A$, and, for $i\in\{1,\dots,r\}$, positive real numbers $\lambda_i$ and complex numbers $\mu_i$. 
The degree of $L(s)$ is the number $d=2\sum_{i=1}^r\lambda_i$, which is independent of the representation for $\Lambda(s)$. 
In the case of automorphic $L$-functions, the numbers $\lambda_i$ may be taken to be $\frac12$, and the number $d$ is a positive integer. 
An $L$-function is described as primitive if it cannot be expressed as a product of lower degree $L$-functions.

This paper is about the simultaneous zeros of two distinct completed $L$-functions $\Lambda_1(s)$ and $\Lambda_2(s)$. 
For $j\in\{1,2\}$, let $d_j$ denote the degree of $\Lambda_j(s)$. 
If $d_2-d_1\leq1$, then, unless $\Lambda_1(s)$ divides $\Lambda_2(s)$ as a completed Euler product, the quotient $\Lambda_2(s)/\Lambda_1(s)$ has infinitely many poles \cite{MM}, \cite{BombieriPerelli}, \cite{Srinivas}, \cite{LFAD}. 
Furthermore, if $d_2-d_1\leq0$, then quantitative bounds were established for the number of poles. 
If $d_2-d_1=2$, then $\Lambda_2(s)/\Lambda_1(s)$ is known to have infinitely many poles only in special cases \cite{ACOZOLF}, \cite{TCTMF}. 
In this paper, we restrict ourselves to Artin $L$-functions such that $d_2-d_1=2$. We will prove:

\begin{theorem}\label{thm.ArtinQuotients}
Let $\phi$ be a $3$-dimensional Artin representation of $\mathrm{Gal}\left(\overline{\mathbb{Q}}/\mathbb{Q}\right)$, let $c$ be complex conjugation, and let $p$ the dimension of the $(+1)$-eigenspace for $\phi(c)$. 
Denote by $L(s,\phi)$ (resp. $\Lambda(s,\phi)$) the Artin $L$-function (resp. completed Artin $L$-function) associated to $\phi$, and by $\zeta(s)$ (resp. $\xi(s)$) the Riemann zeta function (resp. completed Riemann zeta function). 
If $p\geq1$, and $L(s,\phi)$ is primitive, then $\Lambda(s,\phi)/\xi(s)$ has infinitely many poles.
\end{theorem}

We contrast Theorem~\ref{thm.ArtinQuotients} to the well-known theorem that, if $\zeta_K(s)$ is the Dedekind zeta function of Galois extension $K/\mathbb{Q}$, then $\zeta_K(s)/\zeta(s)$ is entire. 
The essential difference is that $\zeta_K(s)$ is in fact divisible by $\zeta(s)$.
The proof of Theorem~\ref{thm.ArtinQuotients} closely follows the strategy used in \cite[Corollary~1.2]{TCTMF}, which states that $\Lambda(\mathrm{Sym}^2f,s)/\xi(s)$ has infinitely many poles for a non-CM Maass form $f$. 
In particular, we argue that, were the quotient to have finitely many poles, it could be identified with an automorphic $L$-function. 
This identification yields a contradiction in the form of an inadmissible linear dependence between Euler products, as we shall see in Section~\ref{sec:mainproof}. 
Arguing this way requires the development of a suitable converse theorem with rather minimal hypotheses, for example, allowing for meromorphic twists and non-standard Euler factors. 
This is achieved in Section~\ref{sec.proof}, following some preliminary calculations in Section~\ref{sec.prelim}. 
In particular, Section~\ref{sec.prelim} develops an explicit analytic continuation of the Gauss hypergeometric function and establishes a connection to twisted $L$-functions.

We conclude this introduction with Examples and Remarks exploring the hypotheses and consequences of Theorem~\ref{thm.ArtinQuotients}. The relevant theory of Artin $L$-functions can be found, for example, in \cite[Chapter~7]{Neu}, or \cite[Section~5]{IK}. We will use the following notation throughout:
\begin{equation}\label{eq.GammaR}
\Gamma_{\R}(s)=\pi^{-s/2}\Gamma\left(\frac{s}{2}\right).
\end{equation}
We note that the completed Riemann zeta function is $\xi(s)=\Gamma_{\R}(s)\zeta(s)$.
Let $\phi$ be as in Theorem~\ref{thm.ArtinQuotients}. 
The representation $\phi$ factors through a number field $F$ and the Artin $L$-function $L(s,\phi)$ is defined by an Euler product of the form:
\begin{equation}\label{ArtinEP}
L(s,\phi)=\prod_{\mathfrak{p}}\prod_{i=1}^{[F:\mathbb{Q}]}(1-\alpha_i\mathrm{Norm}_{F/\mathbb{Q}}(\mathfrak{p})^{-s})^{-1},
\end{equation}
where the product is over the prime ideals $\mathfrak{p}$ in the ring of integers of $F$, and, for $i\in\{1,\dots,[F:\mathbb{Q}]\}$, the complex numbers $\alpha_i$ are either roots of unity or zero.
With $c$ and $p$ as in Theorem~\ref{thm.ArtinQuotients}, let furthermore $m$ denote the dimension of the $(-1)$-eigenspace of $\phi(c)$.
The completed Artin $L$-function is then:
\begin{equation}\label{eq.completedArtin}
\Lambda(s,\phi)=\Gamma_{\mathbb{R}}(s)^{p}\Gamma_{\mathbb{R}}(s+1)^{m}L(s,\phi).
\end{equation}
Note that a $3$-dimensional Artin representation $\phi$ satisfies $p\geq1$ if and only if:
\begin{equation}\label{eq.CancelGamma}
\frac{\Lambda(s,\phi)}{\xi(s)}=L(s)\cdot\begin{cases}
\Gamma_{\mathbb{R}}(s)^2,&p=3,\\
\Gamma_{\R}(s)\Gamma_{\R}(s+1),&p=2,\\
\Gamma_{\R}(s+1)^2,&p=1.
\end{cases}
\end{equation}
for some Dirichlet series $L(s)$. 

\begin{example}\label{ex.31}
Let $\phi$ be as in \cite[\href{http://www.lmfdb.org/ArtinRepresentation}{Artin Representation 3.229.4t5.a.a}]{lmfdb}, that is, the $3$-dimensional Artin representation having the smallest conductor on the LMFDB. The associated Artin $L$-function $L(s,\phi)$ is primitive with gamma factor $\Gamma_{\R}(s)\Gamma_{\R}(s+1)^2$. Theorem~\ref{thm.ArtinQuotients} asserts that infinitely many non-trivial zeros of $\zeta(s)$ are not also zeros of $L(s,\phi)$. 
\end{example}

\begin{remark}
The assumption in Theorem~\ref{thm.ArtinQuotients} that $L(s,\phi)$ be primitive is a useful shorthand, but can be reformulated to taste. 
It is intended as a proxy for the assumption that $\phi$ is an irreducible representation. 
Indeed, conjecturally, the Artin $L$-function attached to an irreducible Galois representation is the $L$-function of some cuspidal automorphic representation (and hence primitive). 
On the other hand, our assumption can be formulated without resorting to primitivity or irreducibility. 
Indeed, $L(s,\phi)$ may be written as an Euler product as per equation~\eqref{ArtinEP}. 
We denote the Euler factor at a prime ideal $\mathfrak{p}$ by $L_{\mathfrak{p}}(s,\phi)$, which is a reciprocal polynomial in $p^{-s}$ where $p$ is a rational prime divisible by $\mathfrak{p}$. 
In the proof of Theorem~\ref{thm.ArtinQuotients} it would suffice that there exists a rational prime $p$ such that, within the polynomial ring $\C[p^{-s}]$, the product $\prod_{\mathfrak{p}\lvert p}L_{\mathfrak{p}}(s,\phi)^{-1}$ indexed by prime ideals of $F$ over $p$ is not divisible by $1-p^{-s}$, that is, the reciprocal Euler factor of the Riemann zeta function at $p$. 
\end{remark}

\begin{remark}
Whilst Artin $L$-functions are not generally known to be automorphic, the Brauer induction theorem implies that they are ratios of Hecke $L$-functions. 
That is, a single Artin $L$-function may itself be written as a quotient of Artin $L$-functions. The weak Artin conjecture claims that the same Artin $L$-function is holomorphic away from a pole at $s=1$ corresponding to the multiplicity of the trivial representation and so in particular has finitely many poles. On the other hand, a theorem of Booker asserts that if the Artin $L$-function of a 2-dimensional Galois representation is not automorphic, then it has infinitely many poles \cite[Corollary]{POALFATSAC}. These claims are not a contradiction to Theorem~\ref{thm.ArtinQuotients}. In fact, we expect the Artin $L$-function of a 2-dimensional Galois representation expressed as a ratio of Hecke $L$-functions to fail the assumptions of Theorem~\ref{thm.ArtinQuotients}. Indeed,  subject to the strong Artin conjecture, all Euler factors would be reciprocal polynomials in $p^{-s}$, which cannot happen under our assumptions. 
\end{remark}

\subsection*{Acknowledgement}
Several computations on this subject were done in collaboration between TO and Michael Neururer, and the present authors are indebted to him for permitting their reproduction here. 
Moreover we are grateful for his reading of a preliminary version of this work, which led to the correction of various errors. 
We thank Masatoshi Suzuki for his comments on an early draft of this work, 
we thank Lejla Smajlovic for her encouragement, 
and we acknowledge the careful reading of this paper by an anonymous referee whose efforts resulted in several improvements in exposition.
TO was supported by the EPSRC through research grant EP/S032460/1.

\section{Preliminaries}\label{sec.prelim}

In this section we introduce hypergeometric functions and twisted $L$-functions.
The relationship between these functions will become apparent in the sequel. 

\subsection{Gauss Hypergeometric function}\label{sec.HGfns}

Let $a$ denote a complex number. For $\mathrm{Re}(a)>0$, we recall the gamma function: 
\begin{equation*}
\Gamma(a)=\int_0^{\infty}e^{-t}t^{a-1}dt.
\end{equation*}
For $\mathrm{Re}(a)\leq0$, $\Gamma(a)$ is defined by meromorphic continuation. In particular, $\Gamma(a)$ has no zeros in $\mathbb{C}$ and simple poles at $a\in\mathbb{Z}_{\leq0}$. For $k\in\mathbb{Z}_{\geq0}$, the Pochhammer symbols are defined by:
\begin{equation*}\label{eq.Pochhammer}
(a)_0=1,\ \ (a)_k
=a(a+1)(a+2)\cdots(a+k-1).
\end{equation*}
For $a\notin\mathbb{Z}_{\leq0}$, we have $(a)_k=\Gamma(a+k)/\Gamma(a)$. The digamma function is defined to be:
\begin{equation*}\label{eq.digamma}
\Psi(a) = \frac{\Gamma'(a)}{\Gamma(a)},
\end{equation*}
and has simple poles at $a\in\mathbb{Z}_{\leq0}$. Recall the following identity\footnote{\href{https://mathworld.wolfram.com/DigammaFunction.html}{mathworld.wolfram.com/DigammaFunction.html}}:
\begin{equation}\label{eq.digammaid}
\frac{1}{a}=\Psi(a+1) -\Psi(a).
\end{equation}
Summing equation~\eqref{eq.digammaid} over $a\in\{1,\dots,k\}$, we get:
\begin{equation}\label{eq.HPsi}
H_k=\Psi(1+k)-\Psi(1),
\end{equation}
where, for $k\in\mathbb{Z}_{\geq1}$, we denote by $H_k$ the $k$th harmonic number, that is, $H_k=\sum_{a=1}^k1/a$.
Recall the Euler--Mascheroni constant:
\begin{equation}\label{eq.PsiGamma}
\gamma = \lim_{k\rightarrow \infty} \left( H_k - \log(k) \right)=-\Psi(1).
\end{equation}
Combining equations \eqref{eq.HPsi} and \eqref{eq.PsiGamma} we deduce:
\begin{equation}\label{eq.HPsiGamma}
	H_k = \Psi(1+k) + \gamma.
\end{equation}
For $k\geq0$, the Pochhammer symbol $(a+\delta)_k$ is a polynomial in $\delta$ with constant term $(a)_k$ and degree $k$. We may therefore write:
\begin{equation}\label{eq.Hkam}
(a+\delta)_k=(a)_k\sum_{m=0}^kH_k(a,m)\delta^m, \ \ H_k(a,0)=1.
\end{equation}
Expanding the product $(a+\delta)_k=(a+\delta)\cdots(a+k-1+\delta),$ we see that:
\begin{equation} \label{eq.expression_for_H_k^{(m)}}
	H_k(a,m) = \sum_{0\leq n_1 < \cdots <n_{k-m} \leq k-1}\frac{\prod_{i=1}^{k-m}(a+n_i)}{(a)_k} 
= \sum_{0 \leq n_1 < \dots < n_m \leq k-1} \frac{1}{\prod_{i=1}^m(a+n_i)}.
\end{equation}
For $k\geq1$, equation~\eqref{eq.expression_for_H_k^{(m)}} with $a=m=1$ implies:
\[
H_k(1,1)=H_k.
\]
On the other hand, replacing $a$ by $a+n$ in equation~\eqref{eq.digammaid}, summing over $n\in\{0,\dots,k-1\}$, and applying equation~\eqref{eq.expression_for_H_k^{(m)}} with $m=1$ we observe:
\begin{equation*}
 H_k(a,1) = \Psi(a+k)-\Psi(a).
\end{equation*}
Consider complex numbers $a,b,c$ such that $c\notin\mathbb{Z}_{\leq0}$. For $|w|<1$, the function $\hyp{a}{b}{c}{w}$ is defined by:
\begin{equation}\label{eq.2F1<1}
\hyp{a}{b}{c}{w}= \sum_{k=0}^\infty w^k\frac{(a)_k(b)_k}{k!(c)_k}.
\end{equation}
The series in equation~\eqref{eq.2F1<1} converges absolutely for $|w|<1$.
For $|w|\geq1$, the function $\hyp{a}{b}{c}{w}$ is defined by analytic continuation.  
For example, if $a-b \notin \mathbb{Z}$ and $|w|>1$, then\footnote{\href{http://functions.wolfram.com/HypergeometricFunctions/Hypergeometric2F1/02/02/0001/}{wolfram.com/HypergeometricFunctions/Hypergeometric2F1/02/02/0001/}}:
\begin{multline}\label{eq.|w|>1}
	\hyp{a}{b}{c}{w} = \left(-w\right)^{-a}\frac{\Gamma\left(b-a\right)\Gamma\left(c\right)}{\Gamma\left(b\right)\Gamma\left(c-a\right)} \sum_{k=0}^{\infty}w^{-k}\frac{\left(a\right)_k\left(a-c+1\right)_k}{k!\left(a-b+1\right)_k} \\
	 + \left(-w\right)^{-b}\frac{\Gamma\left(a-b\right)\Gamma\left(c\right)}{\Gamma\left(a\right)\Gamma\left(c-b\right)} \sum_{k=0}^{\infty}w^{-k}\frac{\left(b\right)_k\left(b-c+1\right)_k}{k!\left(b-a+1\right)_k}.
\end{multline}
Each of the sums in equation~\eqref{eq.|w|>1} converge absolutely for $|w|>1$.
Furthermore\footnote{\href{http://functions.wolfram.com/HypergeometricFunctions/Hypergeometric2F1/02/02/0002/}{wolfram.com/HypergeometricFunctions/Hypergeometric2F1/02/02/0001/}}:
\begin{equation*}
\hyp{a}{a}{c}{w} = \lim_{\delta\rightarrow0} \hyp{a}{a+\delta}{c}{w}.
\end{equation*}
\begin{lemma}
For complex numbers $a,c$ such that $c\notin\mathbb{Z}_{\leq0}$, and a real number $w$ such that $|w|>1$, we have:
\begin{equation}\label{eq.2F1aacw}
\hyp{a}{a}{c}{w} =\lim_{\delta\rightarrow0}\sum_{k=0}^{\infty}\left[\Gamma(\delta)A_k(\delta)+\Gamma(-\delta)B_k(\delta)\right],
\end{equation}
where
\begin{equation}\label{eq.Akd}
A_k(\delta)=w^{-k}(-w)^{-a}\frac{(a)_k(a-c+1)_k\Gamma(c)}{k!(1-\delta)_k\Gamma(a+\delta)\Gamma(c-a)},
\end{equation}
and
\begin{equation}\label{eq.Bkd}
B_k(\delta)=w^{-k}(-w)^{-\delta-a}\frac{(a+\delta)_k(a-c+1+\delta)_k\Gamma(c)}{k!(1+\delta)_k\Gamma(a)\Gamma(c-a-\delta)}.
\end{equation}
\end{lemma}
\begin{proof}
Replacing $b$ by $a+\delta$ in equation~\eqref{eq.|w|>1}, we get:
\begin{multline*}
\hyp{a}{a+\delta}{c}{w}=\sum_{k=0}^{\infty}\Gamma(\delta)w^{-k}(-w)^{-a}\frac{(a)_k(a-c+1)_k\Gamma(c)}{k!(1-\delta)_k\Gamma(a+\delta)\Gamma(c-a)} \\
+\sum_{k=0}^{\infty}\Gamma(-\delta)w^{-k}(-w)^{-\delta-a}\frac{(a+\delta)_k(a-c+1+\delta)_k\Gamma(c)}{k!(1+\delta)_k\Gamma(a)\Gamma(c-a-\delta)}.
\end{multline*}
Taking the limit as $\delta\rightarrow0$, we deduce equation~\eqref{eq.2F1aacw}.
\end{proof}
\begin{lemma}
For integers $k\geq0$ and $n\geq1$, we have:
\begin{equation}\label{eq.nkest}
(n)_k=O\left(k^n(k-1)!\right),
\end{equation}
in which the constant is independent of $k$.
\end{lemma}
\begin{proof}
Note first that $(1)_k=k!$ and so equation~\eqref{eq.nkest} is trivial in this case. We proceed by induction on $n$. To that end, assume that $(n)_k=O\left(k^n(k-1)!\right)$ and compute:
\begin{multline*}
(n+1)_k=(n+1)(n+2)\cdots(n+k)=\frac{n(n+1)(n+2)\cdots(n+k)}{n}\\
=\left(\frac{n+k}{n}\right)(n)_k=O\left(k^{n+1}(k-1)!\right).
\end{multline*}
\end{proof}
Let $D_0(1/2)$ denote the disc in the complex plane with centre $0$ and radius $1/2$.
\begin{lemma}\label{lem.ABhol}
For complex numbers $a,c$ such that $c\not\in\mathbb{Z}_{\leq0}$, and $w\in\mathbb{R}_{<-1}$, the sums $\sum_{k=0}^{\infty}A_k(\delta)$ and $\sum_{k=0}^{\infty}B_k(\delta)$ converge uniformly to holomorphic functions on $D_0(1/2)$.
\end{lemma}
In light of Lemma~\ref{lem.ABhol}, we will write:
\begin{equation}\label{eq.AdBd}
A(\delta)=\sum_{k=0}^{\infty}A_k(\delta), \ \ B(\delta)=\sum_{k=0}^{\infty}B_k(\delta).
\end{equation}
\begin{proof}
Consider first $A(\delta)$. Since the gamma function is non-zero on the complex plane, there exists $X_1\in\mathbb{R}$ such that, for all $\delta\in D_0(1/2)$, we have:
\begin{equation}\label{eq.mod1Gad}
\left\lvert\frac{1}{\Gamma(a+\delta)}\right\rvert\leq X_1.
\end{equation}
For all $k\geq1$
and $\delta\in D_0(1/2)$, we observe:
\begin{equation}\label{eq.mod11-dk}
\left\lvert\frac{1}{(1-\delta)_k}\right\rvert\leq\frac{1}{(1/2)_k}\leq\frac{2}{(k-1)!}.
\end{equation}
Consider $n_1\in\mathbb{Z}_{>0}$ (resp. $n_2\in\mathbb{Z}_{>0}$) such that $n_1\geq|a|$ (res. $n_2\geq|a-c+1|$) and subsequently $\left\lvert(a)_k\right\rvert\leq(n_1)_k$ (resp. $\left\lvert(a+c-1)_k\right\rvert\leq(n_2)_k$).
Applying equation \eqref{eq.nkest} with $n\in\{n_1,n_2\}$, we deduce that:
\begin{equation}\label{eq.n1n2}
(a)_k=O\left(k^{n_1}(k-1)!\right),\ \ (a-c+1)_k=O\left(k^{n_2}(k-1)!\right).
\end{equation}
Substituting equations \eqref{eq.mod1Gad}, \eqref{eq.mod11-dk}, and \eqref{eq.n1n2} into equation \eqref{eq.AdBd}, we deduce, for all $k\geq1$ and $\delta\in D_0(1/2)$, that:
\begin{equation*}
\left\lvert A_k(\delta)\right\rvert=O\left(|w|^{-k}k^{n_1+n_2-1}\right),
\end{equation*}
in which the implied constant is independent of $k$. 
Noting that, since $|w|>1$, the sum $\sum_{k=0}^{\infty}|w|^{-k}k^{n_1+n_2-1}$ converges,
the Weierstrass $M$-test thus implies that $A(\delta)$ converges absolutely and uniformly on $D_0(1/2)$. Since each $A_k(\delta)$ is holomorphic on $D_0(1/2)$, Morera's theorem implies that $A(\delta)$ is holomorphic.

We now consider $B(\delta)$. Similarly to equation~\eqref{eq.mod1Gad}, we see that there is some $X_2$ such that, for all $\delta\in D_0(1/2)$, we have:
\begin{equation}\label{eq.mod1Gcda}
\left\lvert\frac{1}{\Gamma(c-a-\delta)}\right\rvert\leq X_2.
\end{equation}
Similarly to equation~\eqref{eq.mod11-dk}, for all $k\geq1$ and $\delta\in D_0(1/2)$, we have:
\begin{equation}\label{eq.mod11+dk}
\left\lvert\frac{1}{(1+\delta)_k}\right\rvert\leq\frac{2}{(k-1)!}.
\end{equation}
Similarly to equation~\eqref{eq.n1n2}, we furthermore note that, for all $k\geq1$ and $\delta\in D_0(1/2)$, there are $n_3,n_4\in\mathbb{Z}_{>0}$ so that:
\begin{equation}\label{eq.adkest}
(a+\delta)_k=O\left(k^{n_3}(k-1)!\right),\ \ (a-c+1+\delta)_k=O\left(k^{n_4}(k-1)!\right),
\end{equation}
Since $\left\lvert(-w)^{-\delta}\right\rvert=O(1)$, substituting equations \eqref{eq.mod1Gcda}, \eqref{eq.mod11+dk}, and \eqref{eq.adkest} into equation~\eqref{eq.AdBd}, we see that:
\begin{equation*}
\left\lvert B_k(\delta)\right\rvert=O\left(|w|^{-k}k^{n_3+n_4-1}\right),
\end{equation*}
in which the implied constant is independent of $k$.
Noting that the sum $\sum_{k=0}^{\infty}|w|^{-k}k^{n_3+n_4-1}$ converges, the conclusion follows from the Weierstrass $M$-test and Morera's theorem as before.
\end{proof}
\begin{lemma}
For $\delta\in D_0(1/2)$ and $k\geq0$, we have:
\begin{equation}\label{eq.a=1}
\frac{1}{(1\pm\delta)_k}-\frac{1}{k!}\left(1 \mp H_k \delta\right) =  O\left(\delta^2\right),
\end{equation}
in which the implied constant depends on $k$.
\end{lemma}
\begin{proof}
For $h\neq-1$, we observe:
\begin{equation}\label{eq.idh}
	 \frac{1}{1+h} = 1-h + \frac{h^2}{1+h}.
\end{equation}
Combining equations~\eqref{eq.Hkam} and~\eqref{eq.idh} with $h=\sum_{m=1}^k(\pm1)^mH_k(1,m)\delta^m$, we deduce:
\begin{multline}\label{eq.11e}
\frac{1}{(1\pm\delta)_k}= \frac{1}{k!} \left(\frac{1}{1+h}\right)=\frac{1}{k!} \left(1 - h  +  \frac{h^2}{1+h}\right)\\
=\frac{1}{k!}\left(1\mp H_k(a,1)\delta\right)+\frac{1}{k!}\left(-\sum_{m=2}^k(\pm1)^mH_k(a,m)\delta^m+\frac{h^2}{1+h}\right).
\end{multline}
Since $|\delta|<1$, we note:
\begin{equation}\label{eq.modsum2kpm1}
\left\lvert-\sum_{m=2}^k(\pm1)^mH_k(a,m)\delta^m\right\rvert=O\left(\delta^2\right),
\end{equation}
in which the implied constant depends on $k$.
On the other hand, for $\delta\in D_0(1/2)$, equations \eqref{eq.mod11-dk} and \eqref{eq.mod11+dk} imply that:
\begin{equation}\label{eq.11pmdk}
\frac{1}{(1\pm\delta)_k}\leq\frac{2}{(k-1)!}=O(1),
\end{equation}
in which the implied constant depends on $k$. Multiplying equation~\eqref{eq.11pmdk} by $h^2=O\left(\delta^2\right)$, we deduce:
\begin{equation}\label{eq.h21+hd2}
\frac{h^2}{1+h}=O\left(\delta^2\right).
\end{equation}
Equation~\eqref{eq.a=1} follows upon substituting equations \eqref{eq.modsum2kpm1} and \eqref{eq.h21+hd2} into equation \eqref{eq.11e}.
\end{proof}
\begin{theorem}\label{thm.finallimit}
For complex numbers $a,c$ such that $c\notin\mathbb{Z}_{\leq0}$, and $w\in\mathbb{R}_{<-1}$, we have:
\begin{multline}\label{final_limit}
	\hyp{a}{a}{c}{w}  = (-w)^{-a}\frac{\Gamma(c)}{\Gamma(a)\Gamma(c-a)} \sum_{k=0}^{\infty}w^{-k}\frac{(a)_k(a-c+1)_k}{(k!)^2}\left(2\Psi(k+1)    \right.\\
\left.- \Psi(a+k) - \Psi(a-c+1+k)+ \Psi(a-c+1) - \Psi(c-a) +\log(-w) \right).
\end{multline}
\end{theorem}
The sum in equation~\eqref{final_limit} converges absolutely for $w<-1$. 
\begin{proof}
Substituting equation \eqref{eq.AdBd} into equation \eqref{eq.2F1aacw}, we get:
\begin{equation}\label{eq.hypMT}
\hyp{a}{a}{c}{w} = \lim_{\delta \rightarrow 0} \left[ \Gamma(\delta) A(\delta) + \Gamma(-\delta) B(\delta) \right].
\end{equation}
Around $\delta=0$, we recall\footnote{\href{https://functions.wolfram.com/GammaBetaErf/Gamma/06/01/01/01/}{wolfram.com/GammaBetaErf/Gamma/06/01/01/01/}}:
\begin{equation}\label{eq.gamasy}
\Gamma(\pm\delta)=\pm\frac{1}{\delta}-\gamma+O(\delta).
\end{equation}
Substituting equation \eqref{eq.gamasy} into equation \eqref{eq.hypMT}, we obtain:
\begin{multline}\label{eq.hypMT2}
\hyp{a}{a}{c}{w} 
=\lim_{\delta \rightarrow 0} \left[ \frac{ A(\delta) - B(\delta)}{\delta} - \gamma ( A(\delta) +B(\delta) ) + O(\delta \cdot (A(\delta) + B(\delta) )) \right]\\
=\lim_{\delta\rightarrow 0}\left(\frac{A(\delta)-B(\delta)}{\delta}\right)-\gamma\lim_{\delta\rightarrow0}\left(A(\delta)+B(\delta)\right).
\end{multline}
For each $k\geq0$, we introduce:
\begin{equation} \label{eq.def_ck}
C_k = A_k(0) = B_k(0) = w^{-k} (-w)^{-a}\frac{(a)_k(a-c+1)_k \Gamma(c)}{(k!)^2\Gamma(a)\Gamma(c-a)}.
\end{equation}
Since $A(\delta)$ and $B(\delta)$ are continuous, we calculate:
\begin{equation}\label{eq.gsumCk}
\gamma\lim_{\delta\rightarrow0}\left(A(\delta)+B(\delta)\right)=\gamma\left(A(0)+B(0)\right)=2\gamma\sum_{k=0}^{\infty}C_k.
\end{equation}
Recall from Lemma \ref{lem.ABhol} that $A(\delta)$ and $B(\delta)$ are holomorphic for $\delta\in D_0(1/2)$. Since $A(0)=B(0)$, we deduce that:
\begin{equation}\label{eq.limderiv}
\lim_{\delta \rightarrow 0} \left(\frac{ A(\delta) - B(\delta)}{\delta}\right) = \lim_{\delta \rightarrow 0} \left(\frac{ A(\delta) - A(0) - B(\delta)+B(0)}{\delta}\right) = A'(0) - B'(0),
\end{equation}
Substituting equations \eqref{eq.gsumCk} and \eqref{eq.limderiv} into equation \eqref{eq.hypMT2}, we see that:
\begin{equation}\label{eq.hypMT3}
\hyp{a}{a}{c}{w} = A'(0)-B'(0) - 2\gamma\sum_{k=0}^{\infty}C_k.
\end{equation}
Around $\delta=0$, we observe:
\begin{equation}\label{eq.Gad}
\frac{1}{\Gamma(a+\delta)} = \frac{1}{\Gamma(a)}\left(1-\delta\Psi(a)\right) + O\left(\delta^2\right).
\end{equation}
We recall from equation~\eqref{eq.a=1} that:
\begin{equation}\label{eq.1mdk}
\frac{1}{(1-\delta)_k}=\frac{1}{k!}\left(1+\delta H_k\right)+O\left(\delta^2\right).
\end{equation}
Substituting equations~\eqref{eq.Gad} and \eqref{eq.1mdk} into equation \eqref{eq.Akd}, we see that:
\begin{equation}\label{eq.AkdCHP}
A_k(\delta)=C_k+\delta\left(H_k-\Psi(a)\right)C_k+O\left(\delta^2\right).
\end{equation}
From equation~\eqref{eq.AkdCHP}, it follows that:
\begin{equation*} \label{eq.first_der_ak}
	A_k'(0) =\lim_{\delta\rightarrow0}\left(\frac{A_k(\delta)-A_k(0)}{\delta}\right)= \left(H_k - \Psi(a)\right)C_k.
\end{equation*}
Since, for $\delta\in D_0(1/2)$, each $A_k(\delta)$ is holomorphic and $\sum_{k=0}^{\infty}A_k(\delta)$ converges uniformly to $A(\delta)$, we may differentiate term-by-term 
to get: 
\begin{equation}\label{eq.DerivA0}
A'(0)=\sum_{k=0}^{\infty}A'_k(0)=\sum_{k=0}^{\infty}\left(H_k-\Psi(a)\right)C_k.
\end{equation}
Replacing $a$ by $c-a$ and $\delta$ by $-\delta$ in equation~\eqref{eq.Gad}, we find:
\begin{equation}\label{eq.Gcad}
\frac{1}{\Gamma(c-a-\delta)} = \frac{1}{\Gamma(c-a)}\left(1+\delta\Psi(c-a)\right) + O\left(\delta^2\right).
\end{equation}
We recall from equation~\eqref{eq.a=1} that:
\begin{equation}\label{eq.a=1+}
\frac{1}{(1+\delta)_k}=\frac{1}{k!}\left(1-\delta H_k\right)+O\left(\delta^2\right).
\end{equation}
From equation~\eqref{eq.Hkam}, we know:
\begin{equation}\label{eq.a+del}
(a+\delta)_k=(a)_k\left(1+\delta(\Psi(a+k)-\Psi(a))\right)+O(\delta^2),
\end{equation}
Replacing $a$ by $a-c+1$ in equation~\eqref{eq.a+del}, we obtain:
\begin{equation}\label{eq.ac-1+del}
(a+c-1+\delta)_k=(a+c-1)_k(1+\delta(\Psi(a-c+1+k)-\Psi(a-c+1)))+O\left(\delta^2\right),
\end{equation}
Substituting equations \eqref{eq.Gcad}, \eqref{eq.a=1+}  \eqref{eq.a+del}, and \eqref{eq.ac-1+del}, into equation~\eqref{eq.Bkd}, it follows that:
\begin{multline}\label{eq.BkdCk}
B_k(\delta)=(-w)^{-\delta}C_k+ (-w)^{-\delta}\delta C_k\left(-H_k+\Psi(a+k)+\Psi(a-c+1+k)\right.\\
\left.-\Psi(a-c+1)+\Psi(c-a)-\Psi(a)\right)+O\left(\delta^2\right).
\end{multline}
Noting that $-w>0$, writing $(-w)^{-\delta}=\exp(-\delta\log(-w))$, and taking the Taylor expansion of the exponential function, we conclude:
\begin{equation}\label{eq.wmd}
(-w)^{-\delta} = 1 - \delta\log(-w) + O\left(\delta^2\right).
\end{equation}
Substituting equation~\eqref{eq.wmd} into equation~\eqref{eq.BkdCk}, we deduce:
\begin{multline}\label{eq.wdBkd}
B_k(\delta) =  C_k+\delta  C_k\left(-H_k+\Psi(a+k)+\Psi(a-c+1+k)\right.\\
\left.-\Psi(a-c+1)
+\Psi(c-a)-\Psi(a)-\log(-w)\right)
+O\left(\delta^2\right).
\end{multline}
From equation~\eqref{eq.wdBkd}, it follows that:
\begin{multline*}
	B_k'(0) =\lim_{\delta\rightarrow0}\left(\frac{B_k(\delta)-B_k(0)}{\delta}\right) =C_k\left(-H_k+\Psi(a+k)+\Psi(a-c+1+k)\right.\\
\left.-\Psi(a-c+1)+\Psi(c-a)-\Psi(a) -\log(-w) \right).
\end{multline*}
Since, for $\delta\in D_0(1/2)$, each $B_k(\delta)$ is holomorphic and $\sum_{k=0}^{\infty}B_k(\delta)$ converges uniformly to $B(\delta)$, we may differentiate term-by-term
to get: 
\begin{multline}\label{eq.derivB0}
B'(0)=\sum_{k=0}^{\infty}C_k\left(-H_k+\Psi(a+k)+\Psi(a-c+1+k)\right.\\
\left.-\Psi(a-c+1)+\Psi(c-a)-\Psi(a) -\log(-w) \right).
\end{multline}
Subtracting equation~\eqref{eq.derivB0} from equation~\eqref{eq.DerivA0} we find:
\begin{multline}\label{eq.1dAkBk}
A'(0)-B'(0)=\sum_{k=0}^{\infty}C_k\left(2H_k-\Psi(a+k)-\Psi(a-c+1+k)+\Psi(a-c+1)\right.\\
\left.-\Psi(c-a)+\log(-w)\right),
\end{multline}
Substituting equation~\eqref{eq.1dAkBk} into equation \eqref{eq.hypMT3}, it follows that:
\begin{multline}\label{eq.fkqdGcwaC1}
\hyp{a}{a}{c}{w} = \sum_{k=0}^{\infty}C_k\left(2H_k-2\gamma-\Psi(a+k)-\Psi(a-c+1+k)+\Psi(a-c+1)\right.\\
\left.-\Psi(c-a)+\log(-w)\right).
\end{multline}
The result follows upon substituting equations~\eqref{eq.HPsiGamma} and~\eqref{eq.def_ck} into equation~\eqref{eq.fkqdGcwaC1}.
\end{proof}

\subsection{Additive twists}\label{sec:addtwists}
Given a Dirichlet character $\psi$ and complex sequences $\{a_n\}_{n=1}^{\infty},\{b_n\}_{n=1}^{\infty}$, we introduce:
\begin{equation}\label{eq:L}
L_f(s,\psi)=\sum_{n=1}^{\infty}\psi(n)a_nn^{-s},~L_g(s,\bar{\psi})=\sum_{n=1}^{\infty}\overline{\psi(n)}b_nn^{-s}.
\end{equation}
If, for some $\sigma\in\R_{>0}$, we have $|a_n|,|b_n|=O\left(n^{\sigma}\right)$, then the Dirichlet series $L_f(s,\psi)$ and $L_g(s,\bar{\psi})$ converge for $\Re(s)>\sigma+1$. 
We highlight that, in equation~\eqref{eq:L}, the subscript $f$ (resp. $g$) corresponds to the Dirichlet coefficients $\{a_n\}_{n=1}^{\infty}$ (resp. $\{b_n\}_{n=1}^{\infty}$). 
This notation is suggestive of what follows, in which $\{a_n\}_{n=1}^{\infty}$ (resp. $\{b_n\}_{n=1}^{\infty}$) emerge as the Fourier--Whittaker coefficients of a Maass form denoted by $f$ (resp. $g$).

Fix $N\in\mathbb{Z}_{>0}$ and a Dirichlet character $\chi$ mod $N$ (not necessarily primitive). 
In the sequel we will assume that, for all primitive Dirichlet characters $\psi$ modulo $q$ coprime to $N$, the functions $L_f(s,\psi)$ and $L_g(s,\bar{\psi})$ admit meromorphic continuation to $\mathbb{C}$, and, for some $\epsilon\in\{0,1\}$, satisfy: 
\begin{equation}\label{eq:FE}
\Lambda_f(s,\psi)=(-1)^{\epsilon-\epsilon_{\psi}}\chi(q)\psi(N)\frac{\tau(\psi)^2}{q}(q^2N)^{\frac12-s}\Lambda_g(1-s,\bar{\psi}),
\end{equation}
where $\tau(\psi)=\sum_{a\mod q}\psi(a)e^{2\pi i\frac{a}{q}}$ is the Gauss sum, the constant $\epsilon_{\psi}\in\{0,1\}$ is chosen such that $\psi(-1)=(-1)^{\epsilon_{\psi}}$, and
\begin{equation}\label{eq:Lambda}
\Lambda_f(s,\psi)=\Gamma_\R(s+[\epsilon+\epsilon_{\psi}])^2L_f(s,\psi),\ \
\Lambda_g(s,\bar{\psi})= \Gamma_\R(s+[\epsilon+\epsilon_{\psi}])^2L_g(s,\bar{\psi}),
\end{equation}
where, for an integer $k$, we denote by $[k]\in\{0,1\}$ the unique choice such that $k\equiv[k]$ mod $2$. 
In the case that $\psi=\textbf{1}$ is the trivial character, we omit it from the notation, that is:
\begin{equation}\label{eq.trivialcharacter}
L_f(s,\textbf{1})=L_f(s),\ \ \Lambda_f(s,\textbf{1})=\Lambda_f(s), \ \ L_g(s,\textbf{1})=L_g(s),\ \ \Lambda_g(s,\textbf{1})=\Lambda_g(s).
\end{equation}
For $r\in \mathbb{Z}_{\geq 0}$, denote by $\cos^{(r)}$ the $r^{\text{th}}$ derivative of $\cos$.  For $\alpha \in \mathbb{Q}^\times$, the additive twist of $L_f(s)$ by $\alpha$ is the Dirichlet series:
\begin{equation}\label{eq.addtwistDir}
L_f\left(s,\alpha,\cos^{(r)}\right) = \sum_{n=1}^\infty \cos^{(r)} \left(2 \pi n \alpha \right) a_n n^{-s}.
\end{equation}
We define the completed additive twists by:
\begin{equation}  \label{eq:additive_twists_defn}
	\Lambda_f \left( s,\alpha,\cos^{(r)}\right) = \Gamma_{\mathbb{R}} \left( s + [\epsilon+r] \right)^2L_f \left(s,\alpha,\cos^{(r)} \right),
\end{equation} 
where $\Gamma_{\mathbb{R}}(s)$ is as in equation~\eqref{eq.GammaR}.
We define $L_f(s,\alpha,\sin)$ (resp. $\Lambda_f(s,\alpha,\sin)$) as in equation~\eqref{eq.addtwistDir} (resp. equation~\eqref{eq:additive_twists_defn}) with $\cos$ replaced by $\sin$, and note that $\Lambda(s,\alpha,\sin)=-\Lambda\left(s,\alpha,\cos^{(1)}\right)$.
Taking the Fourier expansion of a primitive Dirichlet character $\psi$ mod $q$, we deduce
\begin{equation}\label{character_twists_in_terms_of_additive_twists}
		\Lambda_g(s,\psi) = (-i)^{\epsilon_\psi} \frac{\tau(\psi)}{q} \sum_{b \mod q} \bar{\psi}(-b) \Lambda_g \left(s,\frac{b}{q},\cos^{\left(\epsilon_\psi\right)}\right).
\end{equation}
The inverse relationship is
\begin{multline} \label{additive_twists_intermsofcharactertwists}
		\Lambda_f \left( s,\alpha,\cos^{(r)}\right) = \frac{i^r}{q-1} \sum_{\substack{\psi \mod q \\ \psi \neq \psi_0 \\ \psi(-1) = (-1)^r}}\tau(\bar{\psi}) \Lambda_f \left(s,\psi \right) \\
		+ \begin{cases}
			(-1)^{r/2} \left[ \Lambda_f(s) - \frac{q}{q-1} \Lambda_f(s,\psi_0) \right], &r \text{ even}, \\
			0, &r \text{ odd},
		\end{cases}
\end{multline}
where $\psi_0$ denotes the principal character mod $q$.
Combining equation~\eqref{eq:FE} with equation~\eqref{additive_twists_intermsofcharactertwists}, we deduce the following.
\begin{prop}[Proposition 3.4 in \cite{TCTMF}] \label{prop:functional_equation_twists}
Let $\alpha = \frac{a}{q}$ for $a\in\mathbb{Z}$ and $q$ coprime to $N$. 
If $\Lambda_f(s,\psi)$ satisfies functional equation~\eqref{eq:FE} for all primitive Dirichlet characters $\psi$ of conductor $q$, 
then $\Lambda_f \left( s,\alpha,\cos^{(r)}\right)$ satisfies the following functional equation:
\begin{multline} \label{functional_equations_additive_twists}
		\Lambda_f \left( s,\alpha,\cos^{(r)}\right) = (-1)^\epsilon \frac{i^r(q^2N)^{\frac12-s} \chi(q)}{q-1} \sum_{\substack{\psi \mod q \\ \psi \neq \psi_0 \\ \psi(-1) = (-1)^r}}\psi (N \alpha) \tau(\psi) \Lambda_g \left(1-s,\bar{\psi} \right) \\
		+ \begin{cases}
			(-1)^{r/2} \left[ \Lambda_f(s) - \frac{q}{q-1} \Lambda_f(s,\psi_0) \right], &r \text{ even}, \\
			0, &r \text{ odd}.
		\end{cases}
\end{multline}
\end{prop}

\section{Converse theorems}\label{sec.proof}
For real $u\neq0$, we will write $K(u):=K_0(|u|)$, where, for $u>0$, $K_0(u)$ is the $K$-Bessel function:
\begin{equation*}\label{eq.KB}
K_0(u)=\frac12\int_0^\infty e^{-u(t+t^{-1})/2}\frac{dt}{t}.
\end{equation*} 
By $\mathcal{H}$ we denote the upper half-plane.
Consider complex sequences $\{a_n\}_{n=1}^{\infty}$, $\{b_n\}_{n=1}^{\infty}$ such that $|a_n|,|b_n|=O(n^{\sigma})$ for some $\sigma\in\mathbb{R}_{>0}$. 
Given $\epsilon\in\{0,1\}$, define $a_{-n}=(-1)^{\epsilon}a_n$, $b_{-n}=(-1)^{\epsilon}b_n$, and introduce, for $z=x+iy\in\mathcal{H}$:
\begin{equation}\label{eq:FWseries}
f(z)=f_0(z) + \tilde{f}(z),~g(z)=g_0(z) + \tilde{g}(z),
\end{equation}
where:
\begin{equation}\label{eq:tilde}
\begin{split}
\tilde{f}(z)&=\frac{1}{2}\sum_{n\neq 0}a_n\sqrt{y}K(2\pi ny)\exp(2\pi inx),\\
\tilde{g}(z)&=\frac12\sum_{n\neq 0}b_n\sqrt{y}K(2\pi ny)\exp(2\pi inx),
\end{split}
\end{equation}
and, with $\Lambda_f(s)$ and $\Lambda_g(s)$ as in equation~\eqref{eq.trivialcharacter}:
\begin{equation}\label{eq:consteq0}
\begin{split}
f_0(z)&= -\Res_{s=0}\Lambda_f(s) y^{\frac12} + \Res_{s=0}s\Lambda_f(s) y^{\frac12}\log y,\\
g_0(z)&=-\Res_{s=0}\Lambda_g(s) y^{\frac12} + \Res_{s=0}s\Lambda_g(s) y^{\frac12}\log y.
\end{split}
\end{equation} 
In this Section, we will prove:
\begin{theorem}\label{thm:0Converse}
Let $N$ be a positive integer, 
let $\chi$ be a Dirichlet character mod $N$, 
let $\epsilon\in\{0,1\}$, 
and let $a_n,b_n$ be sequences of complex numbers indexed by $n\in\mathbb{N}$ such that $|a_n|,|b_n|=O\left(n^{\sigma}\right)$ for some $0<\sigma<1$. 
For all $q$ relatively prime to $N$, and all primitive Dirichlet characters $\psi$ modulo $q$, define $\Lambda_f(s,\psi)$ and $\Lambda_g(s,\bar{\psi})$ as in equation \eqref{eq:Lambda}. Let $\mathcal{P}$ be a set of odd primes such that $\{p\in\mathcal{P}:p\equiv u~(\text{mod }v)\}$ is infinite for every $u,v\in\mathbb{Z}_{>0}$ with $(u,v)=1$ and $p\nmid N$ for any $p\in\mathcal{P}$. Whenever the conductor $q$ of $\psi$ is either $1$ or a prime in $\mathcal{P}$, assume that $\Lambda_f(s,\psi)$ and $\Lambda_g(s,\bar{\psi})$ continue to meromorphic functions on $\mathbb{C}$, and satisfy equation~\eqref{eq:FE}. If there is a non-zero polynomial $P(s)\in\mathbb{C}[s]$ such that $P(s)\Lambda_f(s)$ continues to an entire function of finite order, then the functions $\Lambda_f(s)$ and $\Lambda_g(s)$ have at most double poles in the set $\{0,1\}$, the functions $f(z)$ and $g(z)$ defined by equation \eqref{eq:FWseries} are weight 0 Maass forms on $\Gamma_0(N)$ of parity $\epsilon$, nebentypus $\chi$ (resp. $\overline{\chi}$) and eigenvalue $\frac{1}{4}$ and, furthermore, $f(z) = g(-1/Nz)$ for all $z\in\HH$.
\end{theorem}
\begin{remark}
Whilst the $L$-functions associated with irreducible 2-dimensional Galois representations satisfy several conditions in Theorem~\ref{thm:0Converse},   
they are not yet known to have only finitely many poles and so we do not draw any conclusions about their automorphy. 
That said, Theorem~\ref{thm:0Converse} combined with the analogue for holomorphic modular forms \cite[Theorem~1.1]{WCTWP} implies \cite[Corollary]{POALFATSAC} which says that if the Artin $L$-function of an irreducible 2-dimensional Galois representation is not automorphic, then it has infinitely many poles. 
\end{remark}
We will prove Theorem \ref{thm:0Converse} by showing that the assumptions imply that the twists are entire and bounded in vertical strips, at which point we may apply:
\begin{theorem}[Theorem~3.1 in \cite{TCTMF}]\label{thm:EntireConverse}
Let $a_n,b_n$ be sequences of complex numbers such that $|a_n|,|b_n|=O\left(n^{\sigma}\right)$ for some $\sigma\in\R$, and let $N$, $\chi$, $\epsilon$, and $\mathcal{P}$ be as in Theorem~\ref{thm:0Converse}. 
Assume that:
\begin{enumerate}
\item  If $\epsilon = 0$, then $\Lambda_f(s)$ and $\Lambda_g(s)$ continue to holomorphic functions of finite order on $\mathbb{C}-\{0,1\}$ which are bounded in vertical strips with at most double poles in the set $\{0,1\}$,
\item If $\epsilon=1$, then $\Lambda_f(s)$ and $\Lambda_g(s)$ continue to entire functions of finite order which are bounded in vertical strips,
\item For all primitive characters $\psi$ of conductor $q\in\mathcal{P}$, the functions $\Lambda_f(s,\psi)$ and $\Lambda_g(s,\bar{\psi})$ continue to entire functions of finite order which are bounded in vertical strips,
\end{enumerate}
and, for all primitive characters $\psi$ of conductor $q\in\mathcal{P}\cup\{1\}$, the functions $\Lambda_f(s,\psi)$ and $\Lambda_g(s,\bar{\psi})$ satisfy equation~\eqref{eq:FE},
then the functions $f(z)$ and $g(z)$ defined by equation \eqref{eq:FWseries} are weight 0 Maass forms on $\Gamma_0(N)$ of parity $\epsilon$, nebentypus $\chi$ (resp. $\overline{\chi}$) and eigenvalue $\frac{1}{4}$ and, furthermore, $f(z) = g(-1/Nz)$ for all $z\in\HH$.
\end{theorem} 
\begin{remark}
The assumption that $0<\sigma<1$ in Theorem~\ref{thm:0Converse} is stronger than the assumption that $0<\sigma$ in Theorem~\ref{thm:EntireConverse}. The assumption made in Theorem~\ref{thm:0Converse} is sufficient for application to Theorem~\ref{thm.ArtinQuotients}, but it can presumably be relaxed with some additional technicalities. Following the argument below, the assumption that $\sigma<1$ can be avoided until after Corollary \ref{cor:sigafter}.
\end{remark}
We begin with the following Lemma:
\begin{lem}[Lemma 4.2 in \cite{TCTMF}] \label{lem:f(z)-g(1/Nz)}
	Make the assumptions of Theorem \ref{thm:0Converse}. Let $z \in \mathcal{H}$ and write $w = \frac{\Re(z)}{\Im(z)}$. We have:
	
	\begin{equation}\label{eq.circint}
	\tilde{f}(z) - \tilde{g} \left( -\frac{1}{Nz}\right)	= \frac{\left(2w \right)^\epsilon}{2 \pi i} \oint \Lambda_f(s) \hyp{\frac{s+\epsilon}{2}}{\frac{s+\epsilon}{2}}{\frac{1}{2}+\epsilon}{-w^2} \mathrm{Im}(z)^{\frac{1}{2}-s}ds,
	\end{equation}
where the integral is taken so that the contour encloses all poles\footnote{By assumption, there is a polynomial $P(s)\in\C[s]$ so that $P(s)\Lambda_f(s)$ is entire. We make no similar assumptions about the twists $\Lambda_f(s,\psi)$, $\psi\neq\textbf{1}$. This will be relevant again in equations~\eqref{eq.I} and~\eqref{eq.Itilde}.} of $\Lambda_f(s)$.
\end{lem}
\begin{remark}
Lemma \ref{lem:f(z)-g(1/Nz)} was originally stated in terms of a parameter $\nu$ which was assumed to be non-zero, though this assumption was never used in the proof and the statement remains valid as stated above. In particular, the integrand in equation~\eqref{eq.circint} is rapidly decaying as $|s|\rightarrow\infty$ in a vertical strip, and this is an essential part of the proof.
\end{remark}
From now on, we consider $\alpha \in \Q_{>0}$, and set $z=\alpha(1+iy)\in\mathcal{H}$.
Equation~\eqref{eq.circint} now reads:
\begin{equation}\label{eq.circintspecial}
\tilde{f}(z) - \tilde{g} \left( -\frac{1}{Nz}\right)=\frac{(2y^{-1})^{\epsilon}}{2\pi i}\oint \Lambda_f(s) \hyp{\frac{s+\epsilon}{2}}{\frac{s+\epsilon}{2}}{\frac{1}{2}+\epsilon}{-y^{-2}}(\alpha y)^{\frac{1}{2}-s}ds.
\end{equation}
Suppose that $y\in(0,1)$. Replacing $w$ by $-y^{-2}$ in equation \eqref{final_limit}, multiplying by $\left(2y^{-1}\right)^{\epsilon}(\alpha y)^{\frac12-s}$, and noting that $2^{\epsilon}\Gamma\left(\frac12+\epsilon\right)=\sqrt{\pi}$ for $\epsilon\in\{0,1\}$, we get:
\begin{multline}\label{eq.2yFsyGlJ}
	\left(2y^{-1}\right)^\epsilon \text{ } \hyp{\frac{s+\epsilon}{2}}{\frac{s+\epsilon}{2}}{\frac{1}{2}+\epsilon}{ -y^{-2}} (\alpha y)^{\frac{1}{2}-s} 
	= \sqrt{\pi}~\alpha^{\frac12-s}\sum_{k=0}^{\infty} \frac{(-1)^k y^{2k+\frac{1}{2}}}{(k!)^2}G_k(s)\\
\cdot\left(-2\log(y)+J_k(s)\right),
\end{multline}
where, for $k\in\mathbb{Z}_{\geq0}$, we define:
\begin{equation}\label{eq:G}
	G_k(s) = \frac{\left(\frac{s+\epsilon}{2} \right)_k \left(\frac{s-\epsilon+1}{2} \right)_k}{\Gamma(\frac{s+\epsilon}{2}) \Gamma(\frac{1-s+\epsilon}{2})},
\end{equation}
and
\begin{multline}\label{eq:K}
J_k(s) = 2 \Psi(k+1)-\Psi \left( \frac{s+\epsilon}{2} + k \right) - \Psi \left( \frac{s-\epsilon+1}{2} + k \right) + \Psi \left( \frac{s-\epsilon+1}{2}\right) \\ - \Psi \left( \frac{1+\epsilon-s}{2} \right). 
\end{multline}
Though it is not reflected in the notation, we mention that the functions defined in equations~\eqref{eq:G} and \eqref{eq:K} depend on $\epsilon$. Note that $G_k(s)$ is entire, and the product $J_k(s)G_k(s)$ has only removable singularities. 
For $0<y<1$, we deduce from equation~\eqref{eq.2yFsyGlJ} that:
\begin{multline}\label{eq.2yFsyGlJ1}
	\left(2y^{-1}\right)^\epsilon \text{ } \hyp{\frac{s+\epsilon}{2}}{\frac{s+\epsilon}{2}}{\frac{1}{2}+\epsilon}{ -y^{-2}} (\alpha y)^{\frac{1}{2}-s} \\
	= \sqrt{\pi}~\alpha^{\frac12-s}\sum_{k=0}^{\ell_0-1} \frac{(-1)^k y^{2k+\frac{1}{2}}}{(k!)^2}G_k(s)\left(-2\log(y)+J_k(s)\right)+O\left(y^{2\ell_0+\frac12}\log(y)\right).
\end{multline}
For $s$ in a fixed compact set, 
the error term can be chosen independently from $s$. 
Integrating equation \eqref{eq.2yFsyGlJ1} around a contour enclosing all poles of $\Lambda_f(s)$, 
we obtain:
\begin{multline}\label{eq.circintexpand}
		\frac{(2y^{-1})^\epsilon}{2\pi i} \oint \Lambda_f(s) \hyp{ \frac{s+\epsilon}{2}}{\frac{s+\epsilon}{2}}{\frac{1}{2} + \epsilon}{ -y^{-2}} (\alpha y)^{\frac{1}{2}-s} ds \\
	= \sum_{k = 0}^{\ell_0 -1} \left( -2\log(y)\mathcal{I}_k(\alpha) + \widetilde{\mathcal{I}}_k(\alpha) \right) y^{ 2k + \frac{1}{2}} + O\left(y^{2\ell_0+\frac12}\log(y)\right),
\end{multline} 
where
\begin{equation}\label{eq.I}
	\mathcal{I}_k(\alpha) = (-1)^k \frac{\sqrt{\pi}}{(k!)^2} \cdot \frac{1}{2 \pi i} \oint  \Lambda_f(s) G_k(s)  \alpha^{\frac{1}{2} -s} ds,
\end{equation}
and
\begin{equation}\label{eq.Itilde}
\widetilde{\mathcal{I}}_k(\alpha) = (-1)^k \frac{\sqrt{\pi}}{(k!)^2} \cdot \frac{1}{2 \pi i} \oint  \Lambda_f(s)J_k(s) G_k(s) \alpha^{\frac{1}{2} - s} ds.
\end{equation}
The functions in \eqref{eq.I} and \eqref{eq.Itilde} depend on $\epsilon$.
\begin{definition}
For any open interval $(a,b) \subset \mathbb{R}$ denote by $\mathcal{M}(a,b)$ the set of meromorphic functions which are holomorphic on $a <\Re (s) < b$ except for at most double poles at each $s\in\mathbb{Z}$, and which are bounded on the strips $\{ s \in \mathbb{C} : \Re (s) \in [ c,d ], \vert \Im (s) \vert \geq 1 \}$ for each compact $[c,d] \subset (a,b)$. 
Furthermore, let $\mathcal{H}(a,b)$ denote the set of $f \in \mathcal{M}(a,b)$ such that $f$ is holomorphic at each $s\in\mathbb{Z}$.
\end{definition}
Given $\alpha\in\mathbb{Q}_{>0}$, from now on we will write $\beta=-1/N\alpha$.
\begin{lem}\label{hol_lem_1}
For $\alpha \in \mathbb{Q}_{> 0}$ and $\ell_0\in\mathbb{Z}_{>0}$, the following function is in $\mathcal{H}\left(\frac32 - 2\ell_0,\infty\right)$:
\begin{multline}\label{eq.M}
		M_\alpha (s) = (N\alpha ^2)^{s-\frac{1}{2}} \sum_{a \in \{ 0, 1\}}i^{-a} \sum_{\substack{t = 0\\
	t \equiv a + \epsilon \text{ mod }2}}^{2\ell_0 -1}\frac{(2\pi i N \alpha)^t\Gamma_{\mathbb{R}}(1-s)^2\Lambda_g \left( s+t,\beta, \cos^{(a)} \right)}{t!\Gamma_{\mathbb{R}}\left(1-s-2~\lfloor t/2\rfloor \right)^2} \\
	- i^{-\epsilon} \pi^\epsilon \Lambda_f \left(s,\alpha,\cos^{(\epsilon)}\right)+ \alpha^{s-\frac{1}{2}} \sum_{k = 0}^{\ell_0-1} \left(\frac{\widetilde{\mathcal{I}}_k(\alpha)}{s+2k}+2\frac{\mathcal{I}_k(\alpha)}{(s+2k)^2} \right),
\end{multline}	
where $\lfloor t/2\rfloor$ denotes the largest integer $m\leq t/2$. 
\end{lem}
\begin{proof}
Replacing $y$ by $\alpha y$ in \cite[Lemma~2.4]{BCK}, for $y \in (0,\frac12]$ and $\ell_0 \in \Z_{>0}$, we get\footnote{Equation~\eqref{eq.BCK} was originally stated for a function constructed using a Maass form, though the proof demonstrates that it remains valid for $g$. The chosen contour $\mathrm{Re}(s')=2$ requires that $\sigma<1$, however could be replaced with $\mathrm{Re}(s')=1+\sigma+\iota$, for any $\iota>0$.}:
\begin{multline}\label{eq.BCK}
		\tilde{g} \left( -\frac{1}{N(\alpha+\alpha i y)} \right) = O \left( y^{2\ell_0-1} \right) + \sum_{a \in \{0,1\}} i^{-a} \sum_{\substack{t = 0 \\ t \equiv a + \epsilon \mod 2}}^{2\ell_0-1} \frac{(2 \pi i N \alpha)^t}{t!} \\
		\cdot \frac{1}{2 \pi i} \int_{\mathrm{Re}(s')=2}\frac{\Gamma_{\mathbb{R}} \left(1- s' \right)^2\Lambda_g \left(s'+t,\beta,\cos^{(a)} \right)}{\Gamma_{\mathbb{R}} \left( 1-s'-2~ \lfloor t/2\rfloor\right)^2} 
\left( \frac{y}{N \alpha} \right)^{\frac{1}{2}-s'}ds'.
\end{multline}
Denote by $\textbf{1}_{(0,1)}$ the indicator function for the interval $(0,1)$. 
For $y>0$, we introduce the following function:
\begin{multline}\label{eq.smallyestimate}
	F_\alpha(y) := \sum_{a \in \{ 0, 1\}}i^{-a} \sum_{\substack{t = 0\\
	t \equiv a + \epsilon \text{ mod }2}}^{2\ell_0 -1} \frac{(2\pi i N \alpha)^t}{t!} \\
	\cdot \frac{1}{2\pi i}\int_{\mathrm{Re}(s')=2}  \frac{\Gamma_{\mathbb{R}}(1-s')^2\Lambda_g \left( s'+t,\beta, \cos^{(a)} \right)}{\Gamma_{\mathbb{R}}\left(1-s'-2 \lfloor t/2\rfloor \right)^2}  \left(\frac{y}{N \alpha} \right)^{\frac{1}{2}-s'}ds' \\
	- \tilde{f}(\alpha + i \alpha y) + \textbf{1}_{(0,1)}(y) \sum_{k = 0}^{\ell_0 -1} \left( -2\log(y)\mathcal{I}_k(\alpha) + \widetilde{\mathcal{I}}_k(\alpha) \right) y^{ 2k + \frac{1}{2}}.
\end{multline}
Using equations~\eqref{eq.circintspecial}, \eqref{eq.circintexpand}, and \eqref{eq.BCK}, and noting that $y^{2\ell_0+\frac12}\log(y)=O\left(y^{2\ell_0-1}\right)$ as $y\rightarrow0^+$, we deduce that $F_{\alpha}(y)=O\left(y^{2\ell_0-1}\right)$ as $y\rightarrow 0^+$.
On the other hand, shifting the contour in equation~\eqref{eq.smallyestimate} to the right, we see that $F_{\alpha}(y)$ decays rapidly as $y\rightarrow\infty$.
Using the theory of Mellin transforms, we deduce that the following integral transform is in $\mathcal{H}\left(\frac32 - 2\ell_0,\infty\right)$:
\[
\mathcal{M}\left(F_{\alpha}\right)(s)=\int_0^{\infty}F_{\alpha}(y)(\alpha y)^{s-\frac12}\frac{dy}{y}.
\]
The result follows upon identification of $\mathcal{M}\left(F_{\alpha}\right)(s)$ with $M_\alpha(s)$, which is achieved as follows. 
Firstly, note that:
\begin{multline}\label{eq.mellinIItilde}
	\int_0^\infty \textbf{1}_{(0,1)}(y) \sum_{k = 0}^{\ell_0 -1} \left( -2\log\left(y\right)\mathcal{I}_k(\alpha) + \widetilde{\mathcal{I}}_k(\alpha) \right) y^{2k+\frac12}(\alpha y)^{s-\frac12} \frac{dy}{y} \\
=\alpha^{s-\frac{1}{2}} \int_0^\infty \textbf{1}_{(0,1)}(y) \sum_{k = 0}^{\ell_0 -1} \left( -2\log\left(y\right)\mathcal{I}_k(\alpha) + \widetilde{\mathcal{I}}_k(\alpha) \right) y^{s+2k}\frac{dy}{y} \\
	= \alpha^{s-\frac{1}{2}} \sum_{k = 0}^{\ell_0-1} \left(\frac{\widetilde{\mathcal{I}}_k(\alpha)}{s+2k}+2\frac{\mathcal{I}_k(\alpha)}{(s+2k)^2} \right).
\end{multline} 
Secondly, using Mellin inversion, we calculate: 
\begin{multline}\label{eq.MellinInverseMellin}
\int_0^{\infty}\left(\sum_{a \in \{ 0, 1\}}i^{-a} \sum_{\substack{t = 0\\ t \equiv a + \epsilon \text{ mod }2}}^{2\ell_0 -1} \frac{(2\pi i N \alpha)^t}{t!}\right. \\
	\left.\cdot \frac{1}{2\pi i}\int_{\mathrm{Re}(s')=2}  \frac{\Gamma_{\mathbb{R}}(1-s')^2\Lambda_g \left( s'+t,\beta, \cos^{(a)} \right)}{\Gamma_{\mathbb{R}}\left(1-s'-2 \lfloor t/2\rfloor \right)^2}  \left(\frac{y}{N \alpha} \right)^{\frac{1}{2}-s'}ds'\right)(\alpha y)^{s-\frac12}\frac{dy}{y}\\
=(N\alpha ^2)^{s-\frac{1}{2}} \sum_{a \in \{ 0, 1\}}i^{-a} \sum_{\substack{t = 0\\ t \equiv a + \epsilon \text{ mod }2}}^{2\ell_0 -1}\frac{(2\pi i N \alpha)^t\Gamma_{\mathbb{R}}(1-s)^2\Lambda_g \left( s+t,\beta, \cos^{(a)} \right)}{t!\Gamma_{\mathbb{R}}\left(1-s-2~\lfloor t/2\rfloor \right)^2} .
\end{multline}
In order to evaluate the remaining integral, we will apply \cite[6.561(16)]{IntegralTables} with $\mu=s-1$, $\nu=0$, and $a=2\pi\alpha$, which is valid for for $\mathrm{Re}(\mu+1\pm\nu)=\mathrm{Re}(s)>0$:
	\begin{equation} \label{eq:integral_bessel_function0}
		4 \int_0^\infty K(2\pi \alpha y)y^s \frac{dy}{y} = \left(\pi\alpha\right)^{-s} \Gamma\left( \frac{s}{2} \right)^2=\alpha^{-s}\Gamma_{\mathbb{R}}(s)^2.
	\end{equation}
For $n>0$, we recall that $a_n=(-1)^{\epsilon}a_{-n}$. Using equations~\eqref{eq:tilde},~\eqref{eq:additive_twists_defn}, and~\eqref{eq:integral_bessel_function0}, for $\mathrm{Re}(s)>1+\sigma>0$, we compute:
\begin{equation}\label{eq.Mellinftilde}
\begin{split}
&\int_0^{\infty}\tilde{f}(\alpha + i \alpha y) (\alpha y)^{s-\frac12}\frac{dy}{y}\\
=&\int_0^{\infty}\frac{1}{2}\sum_{n\neq 0}a_n\sqrt{\alpha y}K(2\pi n\alpha y)\exp(2\pi in\alpha)(\alpha y)^{s-\frac12}\frac{dy}{y}\\
=&\frac{\alpha^{s}}{2}\sum_{n\neq0}a_n\exp(2\pi in\alpha)|n|^{-s}\int_0^{\infty}K(2\pi \alpha y)y^s\frac{dy}{y}\\
=&\frac{\alpha^{s}}{2}\sum_{n=1}^{\infty}a_n\left(\exp(2\pi in\alpha)+(-1)^{\epsilon}\exp(-2\pi in\alpha)\right)n^{-s}\int_0^{\infty}K(2\pi \alpha y)y^s\frac{dy}{y}\\
=&\alpha^{s}i^{-\epsilon}\sum_{n=1}^{\infty}a_n\cos^{(\epsilon)}(2\pi in\alpha)n^{-s}\int_0^{\infty}K(2\pi \alpha y)y^s\frac{dy}{y}\\
=&i^{-\epsilon}\sum_{n=1}^{\infty}a_n\cos^{(\epsilon)}(2\pi in\alpha)n^{-s}\Gamma_{\mathbb{R}}(s)^2\\
=&i^{-\epsilon}\pi^{\epsilon}L_f\left(s,\alpha,\cos^{(\epsilon)}\right)\Gamma_{\mathbb{R}}(s+[\epsilon+\epsilon])^2\\
=&i^{-\epsilon} \pi^\epsilon \Lambda_f \left(s,\alpha,\cos^{(\epsilon)}\right).
\end{split}
\end{equation}
Since the functions in equation~\eqref{eq.Mellinftilde} extend to meromorphic functions on $\mathrm{Re}(s)>\frac32-2\ell_0$, the identity remains valid within this larger domain away from any poles.
Combining equations \eqref{eq.smallyestimate}, \eqref{eq.mellinIItilde}, \eqref{eq.MellinInverseMellin}, and \eqref{eq.Mellinftilde}, we deduce that $M_\alpha(s) = \mathcal{M}\left (F_\alpha\right)(s)$.
\end{proof}
Let $\beta = \frac{u}{v} \in \mathbb{Q}^{\times}$, where $(u,v)=1$ and $v > 0$. 
As in \cite{WCTWP,TCTMF}, we introduce the infinite set 
 \begin{equation*}\label{eq.Tbeta}
 	T_{\beta} := \left\{ \frac{p}{u} \in \mathbb{Q}_{>0}: \space p \equiv u \mod v, \; p \in \mathcal{P}  \right\}.
 \end{equation*} 
The set $T_\beta$ is unbounded by assumption on $\mathcal{P}$. 
An important property is that if $\lambda \in T_\beta$ then $\Lambda_g\left(s,\lambda \beta, \cos^{(r)}\right) = \Lambda_g\left(s,\beta, \cos^{(r)}\right)$. Given $t_0 \in \mathbb{Z}_{\geq 0}$, choose $\ell_0\in\Z_{>0}$ such that $2\ell_0>t_0$. Consider any subset $T_{\beta,M} \subset T_\beta$ such that $ \left\vert T_{\beta,M} \right\vert =M \geq 2\ell_0 > t_0$. By the theory of Vandermonde determinants, for each $\lambda \in T_{\beta,M}$ there exists $c_\lambda \in \mathbb{C}$ such that 
\begin{equation}	\label{lambdas}
	\sum_{\lambda \in T_{\beta,M}} c_\lambda \lambda^{-t} = \delta_{t_0}(t), \; \; t \in \{ 0,1, \dots, M-1 \},
\end{equation}
where $\delta_{t_0}(t)$ is the Kronecker delta function. 
\begin{lem}	\label{lem:holomorphic_function}
Let $\alpha \in \mathbb{Q}_{>0}$, $t_0 \in \mathbb{Z}_{\geq 0}$, $T_{\beta,M}\subset T_{\beta}$ of size $M \geq 2\ell_0 > t_0$, and choose $c_\lambda \in \mathbb{C}$ to satisfy equation~\eqref{lambdas}. 
The following function is in $\mathcal{H}\left(t_0+\frac32-2\ell_0,\infty\right)$:
\begin{multline} \label{long_lambda}
 		i^{-[\epsilon + t_0]}(N \alpha^2)^{s-\frac{1}{2}}\alpha^{-t_0} \frac{(2 \pi i)^{t_0}\Gamma_{\mathbb{R}}(1-s+t_0)^2\Lambda_g \left( s,\beta,\cos^{([\epsilon + t_0])} \right)}{t_0!\Gamma_{\mathbb{R}}\left(1-s+[t_0] \right)^2}  \\
 		- \sum_{\lambda \in T_{\beta,M}}c_\lambda \lambda^{2s-2t_0-1} \Bigg[ (-i \pi)^\epsilon \Lambda_f\left(s-t_0,\alpha \lambda^{-1},\cos^{(\epsilon)}\right) \\
 		- (\lambda^{-1}\alpha)^{s-t_0-\frac{1}{2}}  \sum_{k=0}^{\ell_0-1} \left( \frac{\widetilde{\mathcal{I}}_k(\alpha \lambda^{-1})}{s-t_0+2k} + 2\frac{\mathcal{I}_k(\alpha \lambda^{-1})}{(s-t_0+2k)^2} \right) \Bigg].
\end{multline}
\end{lem}
\begin{proof}
The function in \eqref{long_lambda} is equal to $\sum_{\lambda \in T_{\beta,M}}c_\lambda \lambda^{2s-2t_0-1}M_{\lambda^{-1}\alpha}(s-t_0)$, where $M_{\lambda^{-1}\alpha}(s-t_0)$ is as in equation~\eqref{eq.M}. The result is a consequence of Lemma \ref{hol_lem_1}.  For the first term we apply \eqref{lambdas} as follows:
\begin{multline*}
	\sum_{\lambda \in T_{\beta,M}}c_\lambda \lambda^{2s-1} \left(N \left(\lambda^{-1}\alpha \right)^2 \right)^{s-\frac{1}{2}} \sum_{a \in \{0,1\}}i^{-a} \sum_{\substack{t = 0 \\ t \equiv a + \epsilon \mod 2}}^{2\ell_0-1} \frac{(2 \pi i N \lambda^{-1} \alpha)^t}{t!}\\
\cdot\frac{ \Gamma_{\mathbb{R}}(1-s)^2\Lambda_g \left(s+t,\lambda \beta, \cos^{(a)} \right)}{\Gamma_{\mathbb{R}}\left(1-s-2~\lfloor t/2\rfloor \right)^2}\\
= i^{-[\epsilon + t_0]}\left( N \alpha^2\right)^{s-\frac{1}{2}} \frac{(2 \pi i N \alpha)^{t_0}\Lambda_g \left( s+t_0,\beta,\cos^{([\epsilon + t_0])} \right)\Gamma_{\mathbb{R}}(1-s)^2}{t_0!\Gamma_{\mathbb{R}}\left(1-s-2~\lfloor t_0/2\rfloor\right)^2} .
\end{multline*}
To deduce equation~\eqref{long_lambda}, we use that $t_0 - 2~\lfloor t_0/2\rfloor = [t_0]$.
\end{proof}
For $t\in\mathbb{Z}$, we define the following subset of $\mathcal{M}(a,b)$:
 \[
 	\mathcal{M}_t(a,b) = \{h \in \mathcal{M}(a,b) : h \text{ is holomorphic at } s \in 2\mathbb{Z} +t+1 \}. 
 \]
Taking $M=2\ell_0+1$, Lemma~\ref{lem:holomorphic_function} implies that the following function is in $\mathcal{M}_{t_0}\left(t_0-M+\frac52,\infty\right)$:
\begin{multline}\label{merom_1}
	i^{-[\epsilon + t_0]}(N \alpha^2)^{s-\frac{1}{2}}\alpha^{-t_0} \frac{(2 \pi i)^{t_0}\Gamma_{\mathbb{R}}(1-s+t_0)^2\Lambda_g \left( s,\beta,\cos^{([\epsilon+t_0])} \right) }{t_0!\Gamma_{\mathbb{R}}\left(1-s+ [t_0] \right)^2}  \\
	-(-i \pi)^\epsilon \sum_{\lambda \in T_{\beta,M}}c_\lambda \lambda^{2s-2t_0-1}\Lambda_f \left( s-t_0,\alpha \lambda^{-1},\cos^{(\epsilon)}\right).
\end{multline}
\begin{prop}\label{prop.Atwistsmero}
Let $q \in \mathcal{P} \cup \{1\}$ and let $\beta=\frac{b}{Nq}$ for some $b \in \mathbb{Z}$ such that $(b,Nq) = 1$. Under the assumptions of Theorem \ref{thm:0Converse}, for any $r \in \{0,1\}$, the functions $\Lambda_f\left(s,\beta,\cos^{(r)}\right)$ and $\Lambda_g\left(s,\beta,\cos^{(r)}\right)$ continue to elements of $\mathcal{M}(-\infty,\infty)$.
\label{prop:twisted_mero_1}
\end{prop}
\begin{proof}
Consider first the function $\Lambda_g\left(s,\beta,\cos^{(r)}\right)$. If $b'\equiv -b$ mod $Nq$, then $\Lambda_g\left(s,\beta,\cos^{(r)}\right)=\Lambda_g\left(s,-\frac{b'}{Nq},\cos^{(r)}\right)$. Without loss of generality, we may therefore assume $\beta < 0$ and $-b \in \mathcal{P}$.  

Let $q' \in \mathcal{P}-\{q\}$ satisfy $(b,q') = 1$. Write $\beta' = \frac{b}{Nq'}$ and $\alpha'=-\frac{1}{N\beta}$. Since $\beta$ and $\beta'$ have the same numerator, the intersection $T_\beta \cap T_{\beta'}$ is infinite. For $t_0 \in \mathbb{Z}_{\geq0}$ we can thus choose a set $T_M \subset T_\beta \cap T_{\beta'}$ with $M > t_0$ elements and we can find $c_\lambda \in \C$ such that equation~\eqref{lambdas} is satisfied. Evaluating the function in equation \eqref{merom_1} at $\beta$ and $\beta'$ and taking the difference, we deduce that the following function is an element of $\mathcal{M}_{t_0}\left(t_0-M+\frac52,\infty\right)$:
	
\begin{multline}\label{eq.ggal}
		\frac{\Gamma_{\mathbb{R}}(1-s+t_0)^2}{\Gamma_{\mathbb{R}}\left(1-s+ [t_0] \right)^2}\left( \alpha^{2s-t_0-1}\Lambda_g\left(s,\beta,\cos^{([\epsilon+t_0])} \right) - \alpha'^{2s-t_0-1}\Lambda_g\left(s,\beta',\cos^{([\epsilon+t_0])} \right) \right) \\
		- \frac{i^{-[\epsilon+t_0]}(-i \pi)^\epsilon t_0!}{N^{s-\frac{1}{2}}(2 \pi i)^{t_0}} \sum_{\lambda \in T_M}c_\lambda \lambda^{2s-2t_0-1} \left( \Lambda_f \left(s-t_0,\alpha \lambda^{-1},\cos^{(\epsilon)} \right) - \Lambda_f \left(s-t_0,\alpha' \lambda^{-1},\cos^{(\epsilon)} \right)  \right).
\end{multline}
Note that the poles of $\Gamma_{\R}(s)$ lie in the region $\Re(s) < 1$. By assumption we have that $a_n = O\left(n^\sigma\right)$, for some $\sigma>0$. For all $\lambda\in T_M$, we see that $\Lambda_f\left(s,\alpha \lambda^{-1},\cos^{(\epsilon)}\right)$ is holomorphic for $\Re(s) > \sigma +1$. Using equation~\eqref{functional_equations_additive_twists}, we deduce that the following function is in $\mathcal{H}\left(-\infty,t_0-\sigma\right)$:
\begin{equation*}
	\Lambda_f \left( s-t_0,\alpha \lambda^{-1},\cos^{(\epsilon)} \right) - \Lambda_f \left( s-t_0,\alpha' \lambda^{-1},\cos^{(\epsilon)} \right).
\end{equation*}
For every $t_0 \in \mathbb{Z}_{\geq 0}$, equation~\eqref{eq.ggal} thus implies that the following function continues to an element of $\mathcal{M}_{t_0}\left(t_0-M+\frac52,t_0-\sigma\right)$:
\begin{equation}	\label{gamma_lambda_diff}
	\frac{\Gamma_{\mathbb{R}}(1-s+t_0)^2}{\Gamma_{\mathbb{R}}\left(1-s+[t_0] \right)^2}\left( \alpha^{2s-t_0-1}\Lambda_g\left(s,\beta,\cos^{([\epsilon+t_0])} \right) - \alpha'^{2s-t_0-1}\Lambda_g\left(s,\beta',\cos^{([\epsilon+t_0])} \right) \right).
\end{equation}
The function in equation \eqref{gamma_lambda_diff} is independent of $T_M$. Taking $M$ to be arbitrarily large, it follows that the function in equation \eqref{gamma_lambda_diff} is in $\mathcal{M}_{t_0}\left(-\infty,t_0-\sigma\right)$. We observe that the zeros of the following function have order $2$ and are contained in the set $2 \mathbb{Z}_{\geq 0} +1+ [t_0]$:
\begin{equation}\label{quotient_gamma}
	\frac{\Gamma_{\mathbb{R}}(1-s+t_0)^2}{\Gamma_{\mathbb{R}}\left(1-s+[t_0] \right)^2}.
\end{equation}
Dividing the function in equation~\eqref{gamma_lambda_diff} by that in \eqref{quotient_gamma}, we conclude that the following function is in $\mathcal{M}\left(-\infty,t_0-\sigma\right)$:
\begin{equation*}
 \alpha^{2s-t_0-1}\Lambda_g\left(s,\beta,\cos^{([\epsilon+t_0])} \right) - \alpha'^{2s-t_0-1}\Lambda_g\left(s,\beta',\cos^{([\epsilon+t_0])} \right).
\end{equation*}
Since $\alpha \neq \alpha'$, varying $t_0 \geq 0$ implies that $\Lambda_g\left( s,\beta,\cos^{\left([\epsilon + t_0] \right)}\right)$ continues to an element in $\mathcal{M}\left(-\infty,t_0-\sigma\right)$. Since $[t_0+\epsilon]$ depends only on the parity $[t_0]$, taking $t_0$ arbitrarily large whilst keeping $[t_0]$ constant implies that $\Lambda_g\left( s,\beta,\cos^{\left([\epsilon + t_0] \right)}\right)$ is in $\mathcal{M}(-\infty,\infty)$. Reversing the roles of $f$ and $g$ and repeating the argument above, we find the same for $\Lambda_f\left( s,\beta,\cos^{\left([\epsilon + t_0] \right)}\right)$.
\end{proof}

\begin{cor}\label{cor:sigafter}
Let $q \in \mathcal{P} \cup \{1\}$ and let $\beta=\frac{b}{q}$ for some $b \in \mathbb{Z}$ such that $(b,q) = 1$. Under the assumptions of Theorem \ref{thm:0Converse}, for any $r \in \{0,1\}$, the functions $\Lambda_f \left(s,\beta,\cos^{(r)}\right)$ and $\Lambda_g \left(s,\beta,\cos^{(r)}\right)$ continue to elements of $\mathcal{M}(-\infty,\infty)$.
	\label{cor:twisted_mero}
\end{cor}

\begin{proof}
As in the proof of Proposition \ref{prop:twisted_mero_1}, we may assume that $\beta <0$ and $-b \in \mathcal{P}$. 
Let $t_0 \in \mathbb{Z}_{\geq0}$, $M>t_0$ and consider $T_{\beta,M} \subset T_\beta$ of cardinality $M$ satisfying \eqref{lambdas}. We have $\alpha = -\frac{q}{Nb}$ and so, if $\lambda \in T_{\beta,M}$, then $\alpha \lambda^{-1} = -\frac{q}{Np}$. Applying Proposition \ref{prop:twisted_mero_1}, we see that $\Lambda_f \left( s-t_0,\alpha \lambda^{-1},\cos^{(\epsilon)}\right)$ continues to an element in $\mathcal{M}(-\infty,\infty)$. By equation \eqref{merom_1} we know that $\Lambda_g \left(s,\beta,\cos^{([\epsilon+t_0])} \right)$ is in $\mathcal{M}\left(t_0-M+\frac52,\infty\right)$ for $t_0 \in \{0,1\}$. Taking $M$ to be arbitrarily large, we conclude that $\Lambda_g \left(s,\beta,\cos^{(r)} \right)$ is in $\mathcal{M}(-\infty,\infty)$. Reversing the roles of $f$ and $g$ and repeating the argument above, the same follows for $\Lambda_g \left(s,\beta,\cos^{(r)}\right)$. 
\end{proof}
Taking $b=q=1$ in Corollary \ref{cor:twisted_mero}, we see that $\Lambda_f(s)$ and $\Lambda_g(s)$ are holomorphic away from at most double poles at $s\in\mathbb{Z}$ and bounded in vertical strips. 
On the other hand, recall that the Dirichlet series defining $L_f(s)$ and $L_g(s)$ in equation~\eqref{eq:L} converge for $\Re(s)>1+\sigma$, 
and so $\Lambda_f(s)$ and $\Lambda_g(s)$ are holomorphic for $\Re(s)>1+\sigma$. 
The functional equation~\eqref{eq:FE} moreover implies that $\Lambda_f(s)$ and $\Lambda_g(s)$ are holomorphic for $\Re(s)<-\sigma$. 
In Theorem~\ref{thm:0Converse} we assume that $0<\sigma<1$, and so we deduce that $\Lambda_f(s)$ and $\Lambda_g(s)$ are holomorphic away from at most double poles at $s\in\{0,1\}$.
In particular, we see that assumption (1) in Theorem~\ref{thm:EntireConverse} is valid.

\begin{lem}\label{lem:res_1}
Assume that $\alpha \in \mathbb{Q}_{>0}$ and $\beta = -1/N\alpha$ are such that, for $r\in\{0,1\}$, the functions $\Lambda_g \left( s, \beta, \cos^{(r)}\right)$ and $\Lambda_f \left( s, \alpha, \cos^{(r)}\right)$ continue to elements of $\mathcal{M}(-\infty, \infty)$. For $s_0 \in \mathbb{Z}_{<1}$, choose $t_0 \in \mathbb{Z}_{>1}$ such that $[t_0] = [s_0]$ and write $j = \frac{1}{2}(t_0-s_0)$. Let $T_{\beta,M} \subset T_\beta$ be a set of size $M \geq 2\ell_0 >t_0-s_0+3/2 > t_0-s_0 = 2j>0$ and choose $c_{\lambda}\in \C$ satisfying equation \eqref{lambdas}. If $\epsilon = 1$, then 
\begin{multline}\label{messy_1}
i \pi \sum_{\lambda \in T_{\beta,M}} c_\lambda \lambda^{2s_0-2t_0-1} \Res_{s=s_0} (s-s_0) \Lambda_f \left(s-t_0,\alpha \lambda^{-1},\sin \right)  		\\	
= i^{-[1+t_0]} (N\alpha^2)^{s_0-\frac{1}{2}} \alpha^{-t_0} \frac{(2 \pi i)^{t_0}\Gamma_{\mathbb{R}}(1-s_0+t_0)^2}{t_0!\Gamma_{\mathbb{R}}\left(1-s_0+ [t_0] \right)^2} \Res_{s=s_0} (s-s_0) \Lambda_g \left(s,\beta,\cos^{([1+t_0])} \right)  \\
+ (-1)^j \frac{\sqrt{\pi}}{(j!)^2}\alpha^{s_0-t_0-1} \Big( \delta_0(s_0) G_j'(1) \Res_{s=1}(s-1)\Lambda_f(s) + \delta_0(s_0) G_j(1) \log(\alpha) \Res_{s=1}(s-1)\Lambda_f(s) \\
-G_j(1) \sum_{\lambda \in T_{\beta,M}}c_\lambda \lambda^{s_0-t_0}\log(\lambda) \Res_{s=1}(s-1)\Lambda_f(s) + G_j(1) \delta_0(s_0) \Res_{s=1}\Lambda_f(s)  \Big).
\end{multline}
\end{lem}
\begin{proof}
Specifying $\epsilon=1$ in equation \eqref{long_lambda}, recalling that $\Lambda(s,\alpha,\sin)=-\Lambda\left(s,\alpha,\cos^{(1)}\right)$, and multiplying by $(s-s_0)$ we get:
\begin{multline} \label{eq.longlambda1}
 		i^{-[1 + t_0]}(N \alpha^2)^{s-\frac{1}{2}}\alpha^{-t_0} \frac{(2 \pi i)^{t_0}  \Gamma_{\mathbb{R}}(1-s+t_0)^2(s-s_0)\Lambda_g \left( s,\beta,\cos^{([1 + t_0])} \right)}{t_0!\Gamma_{\mathbb{R}}\left(1-s+[t_0] \right)^2}  \\
 		- \sum_{\lambda \in T_{\beta,M}}c_\lambda \lambda^{2s-2t_0-1} \Bigg[ i \pi  (s-s_0)  \Lambda_f\left(s-t_0,\alpha \lambda^{-1},\sin\right) \\
 		- (\lambda^{-1}\alpha)^{s-t_0-\frac{1}{2}} (s-s_0) \sum_{k=0}^{\ell_0-1} \left( \frac{\widetilde{\mathcal{I}}_k(\alpha \lambda^{-1})}{s-t_0+2k} + 2\frac{\mathcal{I}_k(\alpha \lambda^{-1})}{(s-t_0+2k)^2} \right) \Bigg].
\end{multline}
Since $s_0 > t_0 - 2\ell_0 + \frac{3}{2}$, Lemma \ref{lem:holomorphic_function} implies that the function in equation \eqref{eq.longlambda1} is holomorphic at $s = s_0$.
Taking its residue, we deduce:
\begin{multline}\label{eq.resressumc}
	 i\pi\Res_{s = s_0}\left[ \sum_{\lambda \in T_{\beta,M}} c_\lambda \lambda^{2s-2t_0-1} (s-s_0) \Lambda_f \left( s-t_0,\alpha \lambda^{-1},\sin \right) \right] \\
	= \Res_{s=s_0} \left[  i^{-[1 + t_0]}(N \alpha^2)^{s-\frac{1}{2}}\alpha^{-t_0} \frac{(2 \pi i)^{t_0}\Gamma_{\mathbb{R}}(1-s_0+t_0)^2(s-s_0)\Lambda_g \left( s,\beta,\cos^{([1 + t_0])} \right)}{t_0!\Gamma_{\mathbb{R}}\left(1-s+ [t_0] \right)^2} \right] \\
	+ 2\sum_{\lambda \in T_{\beta,M}} c_\lambda (\lambda \alpha)^{s_0-t_0-\frac{1}{2}} \mathcal{I}_j(\alpha \lambda^{-1}),
\end{multline}
where, we have used 
\begin{multline}\label{eq.res2laI}
		\Res_{s=s_0} \left[ (\lambda\alpha)^{s-t_0-\frac{1}{2}} (s-s_0)\sum_{k=0}^{\ell_0-1} \left(\frac{\widetilde{\mathcal{I}}_k(\alpha \lambda^{-1})}{s-t_0+2k}+2\frac{\mathcal{I}_k(\alpha \lambda^{-1})}{(s-t_0+2k)^2} \right) \right] \\
		= 2(\lambda\alpha)^{s_0-t_0-\frac{1}{2}}  \mathcal{I}_j(\alpha \lambda^{-1}),
\end{multline}
which holds because $j\in\mathbb{Z}$ satisfies $0\leq j<\ell_0$, and appears because $\lambda^{2s-2t_0-1}\left(\lambda^{-1}\alpha\right)^{s-t_0-\frac12}=(\lambda\alpha)^{s-t_0-\frac12}$.
As per the discussion following Corollary~\ref{cor:sigafter}, we know that $\Lambda_f(s)$ is holomorphic away from integer points in the range $-1<-\sigma<s<1+\sigma<2$. It follows that:
\begin{multline}\label{residue_eval_1}
\mathcal{I}_j(\alpha \lambda^{-1}) = (-1)^j \frac{\sqrt{\pi}}{(j!)^2} \frac{1}{2\pi i} \oint \Lambda_f(s) G_j(s) (\alpha \lambda^{-1})^{\frac{1}{2}-s} ds \\
= (-1)^j \frac{\sqrt{\pi}}{(j!)^2} \sum_{p \in \{0,1\} } \Res_{s = p} \left[ \Lambda_f(s) G_j(s) (\alpha \lambda^{-1})^{\frac{1}{2}-s}\right], 
\end{multline}
where $G_j(s)$ is as in equation \eqref{eq:G} with $\epsilon=1$. 
Since $G_j(s)(\alpha\lambda^{-1})^{\frac12-s}$ is entire and $\Lambda_f(s)$ has at most double poles at $s \in \{0,1\}$, it follows that $\Lambda_f(s)G_j(s)(\alpha\lambda^{-1})^{\frac12-s}$ has at most a double pole at $s\in\{0,1\}$. In fact, $G_j(s)(\alpha\lambda^{-1})^{\frac12-s}$ has a simple zero at $s = 0$ and so the product $\Lambda_f(s)G_j(s)$ has at most a simple pole at $s=0$. 
For $s$ close to $0$, we note the expansion:
\begin{equation}\label{eq.Gexpansion}
	G_j(s) = \frac{\left(1/2\right)_j(j-1)!}{2\sqrt{\pi}}s+O\left(s^2\right),
\end{equation}
in which the implied constant depends on $j$. Substituting equation \eqref{eq.Gexpansion} into equation \eqref{residue_eval_1}, we deduce:
\begin{multline}\label{I_j_res}
	\mathcal{I}_j(\alpha \lambda^{-1}) = \frac{(-1)^j} {2(j!)^2} \left[ (\alpha \lambda^{-1})^\frac{1}{2} \left(1/2\right)_j(j-1)!\Res_{s=0}s\Lambda_f(s)\right.\\
	 +\sqrt{\pi}~\left(G_j'(1) (\alpha \lambda^{-1})^{-\frac{1}{2}} +G_j(1) \log(\alpha \lambda^{-1}) (\alpha \lambda^{-1})^{-\frac{1}{2}} \right) \Res_{s=1} (s-1)\Lambda_f(s) \\
	 \left. +\sqrt{\pi}~ G_j(1) (\alpha \lambda^{-1})^{-\frac{1}{2}} \Res_{s=1}\Lambda_f(s) \right].
\end{multline}
The functions $(s-s_0)\Lambda_f \left(s-t_0,\alpha \lambda^{-1},\sin \right)$ and $(s-s_0)\Lambda_g \left(s,\beta,\cos^{([1+t_0])} \right)$  have at most simple poles at $s = s_0$. Therefore, substituting equation~\eqref{I_j_res} into equation \eqref{eq.resressumc}, we have:
\begin{multline}\label{almost_res_1}
	i \pi  \sum_{\lambda \in T_{\beta,M}}c_\lambda \lambda^{2s_0-2t_0-1}  \Res_{s=s_0} (s-s_0) \Lambda_f \left(s-t_0, \alpha \lambda^{-1},\sin \right)  \\
	=  i^{-[1+t_0]} (N\alpha^2)^{s_0-\frac{1}{2}} \alpha^{-t_0} \frac{(2 \pi i)^{t_0}\Gamma_{\mathbb{R}}(1-s+t_0)^2}{t_0!\Gamma_{\mathbb{R}}\left(1-s+ [t_0] \right)^2}\Res_{s=s_0} (s-s_0) \Lambda_g \left(s,\beta,\cos^{([1+t_0])} \right) \\
	+ (-1)^j\frac{\left(1/2\right)_j}{j(j!)} \alpha^{s_0-t_0} \sum_{\lambda \in T_{\beta,M}} c_\lambda \lambda^{s_0-t_0-1} \Res_{s=0}s\Lambda_f(s) \\
	+ (-1)^j \frac{\sqrt{\pi}}{(j!)^2}\alpha^{s_0-t_0-1}\sum_{\lambda \in T_{\beta,M}} c_\lambda \lambda^{s_0-t_0}\big[ G_j'(1) \Res_{s=1}(s-1)\Lambda_f(s)\\ 
+ G_j(1) \log(\alpha) \Res_{s=1}(s-1)\Lambda_f(s)-G_j(1) \log(\lambda) \Res_{s=1}(s-1)\Lambda_f(s)	+ G_j(1) \Res_{s=1}\Lambda_f(s)\big].
\end{multline}
Equation \eqref{messy_1} follows from equation~\eqref{almost_res_1} using equation~\eqref{lambdas}, upon noting that $\delta_{t_0}(t_0-s_0)=\delta_0(s_0)$ and that the second term on the right hand side  of \eqref{almost_res_1} vanishes since $-t_0 > s_0-t_0-1 > -2\ell_0$. 
\end{proof}
\begin{lem}\label{lem:arbitrary_value} 
Given $t_0 \in \mathbb{Z}_{\geq0}$ and $\beta\in\mathbb{Q}^{\times}$, consider $\ell_0\in\mathbb{Z}_{>0}$ so that $2\ell_0>t_0$ and $T_{\beta,M} \subset T_\beta$ of cardinality $M\geq2\ell_0 > t_0$. There exists $\lambda_0 \in T_{\beta}$ such that the vectors $\left( \lambda^{-t} \right)_{\lambda \in T_{\beta,M} \cup \{ \lambda_0 \}}$ and $\left( \lambda^{-t_0}\log(\lambda) \right)_{\lambda \in T_{\beta,M} \cup \{ \lambda_0 \}}$ are linearly independent, for $t \in \{ 0, 1, \dots, M-1 \}$. 
\end{lem}
\begin{proof}
Consider any $\lambda_0\in T_{\beta}\backslash\{1\}$. For $t \in \{ 0, 1, \dots, M-1 \}$, consider the $(M+1) \times (M+1)$-matrix with columns given by the vectors $\left( \lambda^{-t} \right)_{\lambda \in T_{\beta,M} \cup \{ \lambda_0 \}}$ and $\left( \lambda^{-t_0}\log(\lambda) \right)_{\lambda \in T_{\beta,M} \cup \{ \lambda_0 \}}$. Expanding along the $\lambda_0-$row we see that the determinant of this matrix has the form:
	\begin{equation}\label{eq.determinant}
		\lambda_0^{-t_0}\log(\lambda_0)c + P(\lambda_0^{-1}),
	\end{equation}
where $c$ is a non-zero constant\footnote{More precisely, $c$ is the determinant of the Vandermonde matrix with columns $\left( \lambda^{-t} \right)_{\lambda \in T_{\beta,M}}$ for $t = 0, \hdots, M-1$.} and $P(x)\in \mathbb{C}[x]$. 
Suppose for a contradiction that the expression in equation~\eqref{eq.determinant} vanishes for all $\lambda_0 \in T_{\beta}\backslash\{1\}$, that is:
\begin{equation} \label{eq.c_in_terms_of_lambda_0}
	-c = \frac{\lambda_0^{t_0} P(\lambda_0^{-1})}{\log(\lambda_0)}, \ \ \lambda_0 \in T_\beta \setminus \{ 1 \}.
\end{equation} 
Since the set $T_{\beta}$ is unbounded, it follows that we may choose $\lambda_0$ to be arbitrarily large. The right hand side of \eqref{eq.c_in_terms_of_lambda_0} will always either tend to $0$ or $ \pm \infty$ as $\lambda_0 \rightarrow \infty$, depending on $P$ and $t_0$. This is a contradiction since $c \neq 0$.
\end{proof}
In particular, for all $z \in \mathbb{C}$, there exists $\lambda_0 \in T_\beta$, $c_{\lambda_0}\in\mathbb{C}$ and $c_{\lambda} \in \mathbb{C}$ associated to each $\lambda\in T_{\beta,M}$ such that 
	
	\begin{equation}\label{eq.modifiedlambda}
				\sum_{\lambda \in T_{\beta,M} \cup \{ \lambda_0 \}} c_\lambda \lambda^{-t_0} \log(\lambda) = z, \ \		\sum_{\lambda \in T_{\beta,M} \cup \{ \lambda_0 \}} c_\lambda \lambda^{-t} = \delta_{t_0}(t), \; \; t \in \{ 0,1, \dots, M-1 \}.
	\end{equation}
\begin{lem}\label{lem:L_functions_at_most_simple_poles}
Make the assumptions of Theorem \ref{thm:0Converse}. If $\epsilon =1$, then $\Lambda_f(s)$ and $\Lambda_g(s)$ have at most simple poles at $s \in \{ 0,1\}$.
\end{lem}
\begin{proof}
By assumption, $\Lambda_f(s)$ and $\Lambda_g(s)$ are holomorphic for $\Re(s)\geq2>1+\sigma$ and $\Re(s)<-\sigma<0$. Using the functional equation~\eqref{eq:FE}, we see it suffices to show that:
\begin{equation}\label{eq.suffices}
	\Res_{s=1}(s-1)\Lambda_f(s) = \Res_{s=1}(s-1)\Lambda_g(s) = 0.
\end{equation}
We will show that $\Res_{s=1}(s-1)\Lambda_f(s) = 0$. Reversing the roles of $f$ and $g$ will yield that $\Res_{s=1}(s-1)\Lambda_g(s) = 0$.

For $q\in\mathcal{P}\cup\{1\}$, consider $\beta = - \frac{1}{Nq}$, so that $\alpha=q$. Given a subset $ T_{\beta,M}\subset T_{\beta}$ containing $M$ elements, choose $c_{\lambda}\in\mathbb{C}$ satisfying equation~\eqref{lambdas}. Take $\lambda = p \in T_{\beta,M}$, so that $\alpha \lambda^{-1} = \frac{q}{p}$.  
For $t_0\in\mathbb{Z}_{>1}$, our assumptions on $\{a_n\}_{n=1}^{\infty}$ and $\sigma$ imply, using equation~\eqref{functional_equations_additive_twists}, that $\Lambda_f(s-t_0,\alpha \lambda^{-1},\sin)$ is holomorphic at $s = 0$. 
 For all $t_0\in\mathbb{Z}_{>1}$ such that $[t_0] = 0$, substituting $s_0 = 0$ into equation \eqref{messy_1} gives:
\begin{multline}\label{no_double_poles}
	iN^{-1/2} \frac{\left(2 \pi i\right)^{t_0}\Gamma_{\mathbb{R}}(1+t_0)^2}{t_0!\Gamma_{\mathbb{R}}\left(1\right)^2} \Res_{s=0} s \Lambda_g \left(s,\beta,\sin\right) \\
	= \left(-1\right)^\frac{t_0}{2} \frac{\sqrt{\pi}}{\left(\frac{t_0}{2}!\right)^2} \Big[  G_{t_0/2}'(1) \Res_{s=1} (s-1)\Lambda_f\left(s\right)  +  G_{t_0/2}(1) \log\left(\alpha\right) \Res_{s=1}  (s-1)\Lambda_f\left(s\right)  \\
	-G_{t_0/2}(1) \sum_{\lambda \in T_{\beta,M}}c_\lambda \lambda^{-t_0}\log\left(\lambda\right) \Res_{s=1}(s-1)\Lambda_f\left(s\right)  + G_{t_0/2}(1) \Res_{s=1} \Lambda_f\left(s\right)   \Big].
	\end{multline}
By Lemma \ref{lem:arbitrary_value}, for any $z \in \C$ we can find $\lambda_0 \in T_\beta$ and $c_\lambda, c_{\lambda_0} \in \C$ such that equation \eqref{eq.modifiedlambda} is satisfied. The proof of Lemma \ref{lem:res_1} remains valid if we replace $T_{\beta,M}$ by $T_{\beta,M} \cup \{ \lambda_0 \}$. In particular, equation \eqref{no_double_poles} holds with $\sum_{\lambda \in T_{\beta,M}}c_\lambda \lambda^{-t_0}\log(\lambda)$ replaced by an arbitrary complex number. 
Since $G_{t_0/2}(1) \neq 0$, we conclude that $\Res_{s=1}(s-1)\Lambda_f(s) = 0$. 
\end{proof}

\begin{lem}\label{lem:residue_2}
Make the assumptions of Lemma \ref{lem:res_1}, so that, in particular, we have $\epsilon=1$ and $[t_0]=[s_0]$. We have:
\begin{multline}\label{residues_simple}
	i \pi \Res_{s=s_0} \left[ \sum_{\lambda \in T_{\beta,M}}c_\lambda \lambda^{2s-2t_0-1} \Lambda_f \left(s-t_0, \alpha \lambda^{-1},\sin \right) \right] \\
	= \Res_{s=s_0} \left[ i^{-[1+t_0]} (N\alpha^2)^{s-1/2} \alpha^{-t_0} \frac{(2 \pi i)^{t_0}\Gamma_{\mathbb{R}}(1-s+t_0)^2}{t_0!\Gamma_{\mathbb{R}}\left(1-s+[t_0] \right)^2}\Lambda_g \left(s,\beta,\cos^{([1+t_0])} \right) \right]	\\
	+ (-1)^j\frac{\sqrt{\pi}}{(j!)^2} G_j(1) \alpha^{s_0-t_0-1} \sum_{\lambda \in T_{\beta,M}}c_\lambda\lambda^{-t_0}\log(\lambda) \Res_{s=1}\Lambda_f(s) 	\\
	+ (-1)^j\frac{\sqrt{\pi}}{(j!)^2} G_j(1) \alpha^{s_0-t_0-1} \log(\alpha) \delta_0(s_0)\Res_{s=1}\Lambda_f(s) \\
	+ (-1)^j \frac{\sqrt{\pi}}{(j!)^2} J_j(1)G_j(1) \alpha^{s_0-t_0-1} \delta_0(s_0) \Res_{s=1}\Lambda_f(s).	
\end{multline}
\end{lem}
\begin{proof}
For $j$ as in Lemma~\ref{lem:res_1}, we compute:
\begin{equation}\label{eq.res412}
	\Res_{s=s_0} \left[ (\lambda \alpha)^{s-t_0-\frac{1}{2}} \sum_{k=0}^{\ell_0-1}\frac{\mathcal{I}_k(\alpha \lambda^{-1})}{(s-t_0+2k)^2}\right]= \log(\lambda \alpha) (\lambda \alpha)^{s_0-t_0-\frac{1}{2}} \mathcal{I}_j(\alpha \lambda^{-1}).
\end{equation}
By Lemma \ref{lem:L_functions_at_most_simple_poles}, the function $\Lambda_f(s)$ has at most simple poles. Simplifying equation~\eqref{I_j_res} accordingly, we get:
\begin{equation}\label{I_j_res_simplified}
	\mathcal{I}_j(\alpha \lambda^{-1}) =(-1)^j(\alpha \lambda^{-1})^{-\frac{1}{2}}  \frac{\sqrt{\pi}~ G_j(1) } {2(j!)^2}\Res_{s=1}\Lambda_f(s).
\end{equation}
Substituting equation \eqref{I_j_res_simplified} into equation~\eqref{eq.res412}, we obtain: 
\begin{multline}\label{eq.TML}
\Res_{s=s_0} \left[ (\lambda \alpha)^{s-t_0-\frac{1}{2}} \sum_{k=0}^{\ell_0-1}\frac{\mathcal{I}_k(\alpha \lambda^{-1})}{(s-t_0+2k)^2}\right] \\
= (-1)^j  \log(\lambda \alpha) \lambda^{s_0-t_0} \alpha^{s_0-t_0-1} \frac{\sqrt{\pi}G_j(1)}{2(j!)^2} \Res_{s=1} \Lambda_f(s).
\end{multline}
On the other hand, we observe:
\begin{equation}\label{eq.TMLL}
	\Res_{s=s_0}\left[ (\lambda \alpha)^{s-t_0-1} \sum_{k=0}^{\ell_0-1} \frac{\widetilde{\mathcal{I}}_k(\alpha \lambda^{-1})}{s-t_0+2k} \right] = (\lambda \alpha)^{s_0-t_0-1} \widetilde{\mathcal{I}}_j(\alpha \lambda^{-1}).
\end{equation}
Mimicking the computation of $\mathcal{I}_j(\alpha\lambda^{-1})$ presented in the proof of Lemma \ref{lem:res_1}, we deduce:
\begin{equation}\label{eq.ItildeRes}
		\widetilde{\mathcal{I}}_j(\alpha \lambda^{-1}) = (-1)^j \frac{\sqrt{\pi}}{(j!)^2} \sum_{p \in \{0,1 \} } \Res_{s = p} \left[ \Lambda_f(s) J_j(s)G_j(s) (\alpha \lambda^{-1})^{\frac{1}{2}-s}\right],
\end{equation}
where $J_j(s)$ is as in equation \eqref{eq:K}, both with $\epsilon=1$. Since $\Lambda_f(s)$ has at most simple poles, equation~\eqref{eq.ItildeRes} implies\footnote{Note that $J_j(s)G_j(s)$ has a removable singularity at $s=0$, and $\lim_{s\rightarrow 0} J_j(s)G_j(s) \neq 0$. Abusing notation, we write $J_j(0)G_j(0) = \lim_{s\rightarrow 0} J_j(s)G_j(s)$.}:
\begin{multline}\label{equationabove}
	\widetilde{\mathcal{I}}_j(\alpha \lambda^{-1}) = (-1)^j \frac{\sqrt{\pi}}{(j!)^2} \left( \alpha^{\frac{1}{2}} \lambda^{-\frac{1}{2}} J_j(0)G_j(0) \Res_{s=0}\Lambda_f(s)\right.\\ 
	\left.+  \alpha^{-\frac{1}{2}} \lambda^{\frac{1}{2}} J_j(1)G_j(1) \Res_{s=1}\Lambda_f(s)  \right).
\end{multline}
Substituting equation~\eqref{equationabove} into equation~\eqref{eq.TMLL}, we get:
\begin{multline}\label{eq.TMLL2}
	\Res_{s=s_0}\left[ (\lambda \alpha)^{s-t_0-1} \sum_{k=0}^{\ell_0-1} \frac{\widetilde{\mathcal{I}}_k(\alpha \lambda^{-1})}{s-t_0+2k} \right] \\
= (-1)^j(\lambda \alpha)^{s_0-t_0-1} \frac{\sqrt{\pi}}{(j!)^2}\left( \alpha^{\frac{1}{2}} \lambda^{-\frac{1}{2}} J_j(0)G_j(0) \Res_{s=0}\Lambda_f(s)+  \alpha^{-\frac{1}{2}} \lambda^{\frac{1}{2}} J_j(1)G_j(1) \Res_{s=1}\Lambda_f(s)  \right).
\end{multline}
The function in equation~\eqref{long_lambda} is holomorphic at $s_0$. Taking the residue of this function at $s=s_0$ and applying equations~\eqref{eq.TML} and \eqref{eq.TMLL2}, we obtain:
\begin{multline}\label{eq.above}
		i \pi \Res_{s=s_0} \left[ \sum_{\lambda \in T_{\beta,M}}c_\lambda \lambda^{2s-2t_0-1} \Lambda_f \left(s-t_0, \alpha \lambda^{-1},\sin \right) \right] \\
	=\Res_{s=s_0} \left[ i^{-[1+t_0]} (N\alpha^2)^{s-1/2} \alpha^{-t_0} \frac{(2 \pi i)^{t_0}\Gamma_{\mathbb{R}}(1-s+t_0)^2}{t_0!\Gamma_{\mathbb{R}}\left(1-s+ [ t_0 ] \right)^2} \Lambda_g \left(s,\beta,\cos^{([1+t_0])} \right) \right]	\\
	+ (-1)^j\frac{\sqrt{\pi}}{(j!)^2} G_j(1) \alpha^{s_0-t_0-1} \sum_{\lambda \in T_{\beta,M}}c_\lambda\lambda^{s_0-t_0}\log(\lambda) \Res_{s=1}\Lambda_f(s) 	\\
	+ (-1)^j\frac{\sqrt{\pi}}{(j!)^2} G_j(1) \alpha^{s_0-t_0-1} \log(\alpha) \sum_{\lambda \in T_{\beta,M}}c_\lambda\lambda^{s_0-t_0} \Res_{s=1}\Lambda_f(s) \\
	+ (-1)^j \frac{\sqrt{\pi}}{(j!)^2} J_j(0)G_j(0) \alpha^{s_0-t_0} \sum_{\lambda \in T_{\beta,M}}c_\lambda \lambda^{s_0-t_0-1} \Res_{s=0}\Lambda_f(s) \\
	+ (-1)^j \frac{\sqrt{\pi}}{(j!)^2} J_j(1)G_j(1) \alpha^{s_0-t_0-1} \sum_{\lambda \in T_{\beta,M}}c_\lambda\lambda^{s_0-t_0} \Res_{s=1}\Lambda_f(s).
\end{multline}
Equation~\eqref{residues_simple} follows from equation~\eqref{eq.above}, using equation~\eqref{lambdas} and noting that $\delta_{t_0}(t_0-s_0)=\delta_0(s_0)$ as in Lemma~\ref{lem:res_1}.
\end{proof}
\begin{prop}\label{prop:epsilon_1_holo}
Make the assumptions of Theorem \ref{thm:0Converse}. If $\epsilon = 1$, then $\Lambda_f(s)$ and $\Lambda_g(s)$ are entire and bounded in vertical strips.
\end{prop}
That is, assumption (2) in Theorem~\ref{thm:EntireConverse} is valid.
\begin{proof} 
Since $\Lambda_f(s)$ and $\Lambda_g(s)$ are in $\mathcal{M}(-\infty,\infty)$, it suffices to prove entirety. 
Mimicking the argument surrounding equation~\eqref{eq.suffices}, we see it suffices to show that $\Res_{s=1}\Lambda_f(s)=0$ (reversing the roles of $f$ and $g$ will yield that $\Res_{s=1}\Lambda_g(s) = 0$).

Consider $t_0\in\mathbb{Z}_{>1}$ such that $[t_0]=0$, 
let $\alpha\in\mathbb{Q}_{>0}$ and $\beta=-1/N\alpha$ be such that, for $r\in\{0,1\}$, the functions $\Lambda_g \left( s, \beta, \cos^{(r)}\right)$ and $\Lambda_f \left( s, \alpha, \cos^{(r)}\right)$ continue to elements of $\mathcal{M}(-\infty, \infty)$, 
let $M>t_0+3/2$, 
let $T_{\beta,M}$ be a set of size $M$,
and, for $\lambda\in T_{\beta,M}$, choose $c_{\lambda}$ satisfying equation~\eqref{lambdas}.
We may apply Lemma \ref{lem:residue_2} in the case that $s_0=0$.
For $\lambda\in T_{\beta,M}$, our assumptions on $t_0$, $\{a_n\}_{n=1}^{\infty}$ and $\sigma$ imply, using equation~\eqref{functional_equations_additive_twists}, that $\Lambda_f(s-t_0,\alpha \lambda^{-1},\sin)$ is holomorphic at $s = 0$. 
Taking $s_0=0$, equation~\eqref{residues_simple} simplifies to:
\begin{multline}\label{no_poles}
-\Res_{s=0} \left[ i N^{-1/2} \frac{(2 \pi i)^{t_0}\Gamma_{\mathbb{R}}(1-s+t_0)^2}{t_0!\Gamma_{\mathbb{R}}\left(1-s+ [ t_0 ] \right)^2} \Lambda_g \left(s,\beta,\cos^{([1+t_0])} \right) \right]	\\
	= (-1)^\frac{t_0}{2}\frac{\sqrt{\pi}}{(\frac{t_0}{2}!)^2} G_{t_0/2}(1)  \sum_{\lambda \in T_{\beta,M}}c_\lambda\lambda^{-t_0}\log(\lambda) \Res_{s=1}\Lambda_f(s) 	\\
	+ (-1)^\frac{t_0}{2}\frac{\sqrt{\pi}}{(\frac{t_0}{2}!)^2} G_{t_0/2}(1)  \log(\alpha) \Res_{s=1}\Lambda_f(s) \\
	+ (-1)^\frac{t_0}{2} \frac{\sqrt{\pi}}{(\frac{t_0}{2}!)^2} J_{t_0/2}(1)G_{t_0/2}(1)\Res_{s=1}\Lambda_f(s).
\end{multline}
We conclude that $\Res_{s=1}\Lambda_f(s) = 0$ via an argument analogous to that presented in the final paragraph in the proof of Lemma~\ref{lem:L_functions_at_most_simple_poles}, using \eqref{no_poles} in place of equation~equation~\eqref{no_double_poles}. 
\end{proof}
\begin{lem}\label{lem:double_pole_residue_epsilon_0}
Let $\alpha,\beta,r,s_0,t_0,j,M,c_{\lambda}$ be as in Lemma~\ref{lem:res_1}. If $\epsilon = 0$, then
\begin{multline}\label{double_pole_residue_epsilon_0}
	\sum_{\lambda \in T_{\beta,M}} c_\lambda \lambda^{2s_0-2t_0-1} \Res_{s=s_0} (s-s_0) \Lambda_f \left(s-t_0,\alpha \lambda^{-1},\cos \right)  		\\	
	= i^{-[t_0]} (N\alpha^2)^{s_0-1/2} \alpha^{-t_0} \frac{(2 \pi i)^{t_0}\Gamma_{\mathbb{R}}(1-s_0+t_0)^2}{t_0!\Gamma_{\mathbb{R}}\left(1-s_0+ [ t_0 ] \right)^2} \Res_{s=s_0} (s-s_0) \Lambda_g \left(s,\beta,\cos^{([t_0])} \right)  \\
+ \delta_0(s_0)(-1)^{j+1} \frac{(1/2)_j}{j!}\alpha^{s_0-t_0-1} \Res_{s=1} (s-1) \Lambda_f(s) .
\end{multline}
\end{lem}
\begin{proof}
The proof is similar to that of Lemma \ref{lem:res_1}. In place of equation~\eqref{eq.resressumc}, we get:
\begin{multline*}
	\Res_{s = s_0}\left[ \sum_{\lambda \in T_{\beta,M}} c_\lambda \lambda^{2s-2t_0-1} (s-s_0) \Lambda_f \left( s-t_0,\alpha \lambda^{-1},\cos \right) \right] \\
	= \Res_{s=s_0} \left[  i^{-[t_0]}(N \alpha^2)^{s-\frac{1}{2}}\alpha^{-t_0} \frac{(2 \pi i)^{t_0}\Gamma_{\mathbb{R}}(1-s+t_0)^2(s-s_0)\Lambda_g \left( s,\beta,\cos^{([t_0])} \right)}{t_0!\Gamma_{\mathbb{R}}\left(1-s+[t_0]\right)^2}\right] \\
	+ 2\sum_{\lambda \in T_{\beta,M}} c_\lambda (\lambda \alpha)^{s_0-t_0-\frac{1}{2}} \mathcal{I}_j(\alpha \lambda^{-1}).
	\end{multline*}
We may compute $\mathcal{I}_j(\alpha\lambda^{-1})$ as in equation~\eqref{residue_eval_1}. When $\epsilon=0$, the function $G_j(s)$ has a double zero at $s=0$, and a single zero at $s=1$. For $s$ close to $1$, we have the expansion:
\begin{equation*}\label{eq.omg}
	\frac{1}{\Gamma\left(\frac{1-s}{2}\right)} = - \frac{s-1}{2}  + O\left( \left(\frac{1-s}{2}\right)^2\right),
\end{equation*}
and so:
\begin{equation} \label{eq.I_twiddle_alpha_lambda}
		\mathcal{I}_j(\alpha \lambda^{-1}) = \frac{(-1)^{j+1}(1/2)_j}{2(j!)} (\alpha \lambda^{-1})^{-\frac{1}{2}} \Res_{s=1} (s-1) \Lambda_f(s).
\end{equation}
We conclude as in the proof of Lemma~\ref{lem:res_1}, noting again that $\delta_{t_0}(t_0-s_0)=\delta_0(s_0)$.
\end{proof}
\begin{lem}\label{lem:twistssimplepoles}
Make the assumptions of Theorem~\ref{thm:0Converse}. If $\psi$ is a primitive Dirichlet character with conductor $q \in \mathcal{P}$, then the twisted $L$-functions $\Lambda_f(s,\psi)$ and $\Lambda_g(s,\psi)$ extend to holomorphic functions except for at most simple poles in the set $\{0,1\}$.
\label{lem:character_only_simple}
\end{lem}

\begin{proof}
Mimicking the argument surrounding equation~\eqref{eq.suffices}, we see it suffices to show that $\Res_{s = 0}  s \Lambda_g(s,\psi)=0$ (reversing the roles of $f$ and $g$ will yield that $\Res_{s=0}s\Lambda_f(s,\psi) = 0$).

Given $s_0 \in \mathbb{Z}_{<1}$, choose $\ell_0,M,t_0\in\mathbb{Z}_{>1}$ such that $M\geq2\ell_0>t_0-s_0+\frac32$. 
For $\alpha\in\mathbb{Q}_{>0}$ and $\beta=-1/N\alpha$ choose $T_{\beta,M}\subset T_{\beta}$ of size $M$, and, for each $\lambda\in T_{\beta,M}$, choose $c_{\lambda}\in\mathbb{C}$ satisfying equation~\eqref{lambdas}. 
Since $s_0 > t_0-2\ell_0+\frac32$ the function in equation \eqref{long_lambda} is holomorphic at $s=s_0$. 
Multiplying this function by $(s-s_0)$ and taking the residue at $s = s_0$, we see that:
\begin{multline}\label{odd_residues}
	i^{-[\epsilon + t_0]}(N \alpha^2)^{s_0-\frac{1}{2}}\alpha^{-t_0} \frac{(2 \pi i)^{t_0}\Gamma_{\mathbb{R}}(1-s_0+t_0)^2}{t_0!\Gamma_{\mathbb{R}}\left(1-s_0+ [ t_0 ] \right)^2}  \Res_{s=s_0} (s-s_0) \Lambda_g \left( s,\beta,\cos^{([\epsilon + t_0])} \right)  \\
	=  \sum_{\lambda \in T_{\beta,M}} c_\lambda \lambda^{2s_0-2t_0-1} (-i\pi)^\epsilon \Res_{s = s_0}(s-s_0) \Lambda_f \left( s-t_0,\alpha \lambda^{-1},\cos^{(\epsilon)} \right). 
\end{multline}
The remainder of the proof is split into five steps.

\subsubsection*{Step 1: Additive twists ($\epsilon = 1$)} 
First set $\beta = \frac{b}{Nq}$, for some $b<0$ such that $(b,Nq)=1$.  
For $\lambda\in T_{\beta}$, it follows that $\alpha \lambda^{-1} = \frac{q}{p}$ for some $p \equiv b \mod Nq$. 
Using our assumptions on $s_0$, $t_0$, $\{a_n\}_{n=1}^{\infty}$, and $\sigma$, we deduce from equation~\eqref{functional_equations_additive_twists} that $\Lambda_f(s-t_0,\alpha \lambda^{-1},\sin)$ is holomorphic at $s = s_0$. 
It follows that the expression in the right-hand side of equation \eqref{odd_residues} is zero. 
Choosing $t_0$ of different parity to $s_0$, we get $[\epsilon+t_0]=[s_0]$. Therefore $\Res_{s = s_0} (s-s_0) \Lambda_g \left( s,\frac{b}{Nq}, \cos^{([s_0])} \right) = 0$ for all $s_0 < 1$. 

If instead $b>0$, and still satisfying $(b,Nq)=1$, then we can take $b' \in \mathcal{P}$ with $b' \equiv -b \mod Nq$, so that we have $\Lambda_g \left(s,\frac{b}{Nq},\cos^{([s_0])} \right) = \Lambda_g \left(s,-\frac{b'}{Nq},\cos^{([s_0])} \right)$. 
Combining the cases $b<0$ and $b>0$, we observe that $\Res_{s = s_0} (s-s_0) \Lambda_g \left( s,\beta, \cos^{([s_0])} \right) = 0$ for all $\beta = \frac{b}{Nq}$ with $(b,Nq)=1$. 

Taking $t_0$ of the same parity as $s_0$, we may apply equation \eqref{messy_1} and Proposition \ref{prop:epsilon_1_holo} to obtain
\begin{multline}\label{even_residues_e_1}
		i^{-[1+t_0]} (N\alpha^2)^{s_0-\frac12} \alpha^{-t_0} \frac{(2 \pi i)^{t_0}\Gamma_{\mathbb{R}}(1-s_0+t_0)^2}{t_0!\Gamma_{\mathbb{R}}\left(1-s_0+ [ t_0 ] \right)^2}\Res_{s=s_0} (s-s_0) \Lambda_g \left(s,\beta,\cos^{([1+t_0])} \right)  \\
		=i \pi \sum_{\lambda \in T_{\beta,M}} c_\lambda \lambda^{2s_0-2t_0-1} \Res_{s=s_0} (s-s_0) \Lambda_f \left(s-t_0,\alpha \lambda^{-1},\sin \right).
\end{multline}
If $\beta = \frac{b}{Nq}$, for some $b\in\mathbb{Z}$ such that $(b,Nq)=1$, then, using equation~\eqref{functional_equations_additive_twists}, we conclude as above that the right-hand side of equation \eqref{even_residues_e_1} vanishes. 
Hence, for all $s_0 < 1$, we have $\Res_{s=s_0}(s-s_0)\Lambda_g \left( s, \frac{b}{Nq},\cos^{([1+s_0])} \right) = 0$. 

On the other hand, if $\beta = \frac{b}{q}$ for some $b<0$ such that $(b,q)=1$, then $\lambda^{-1}\alpha = \frac{q}{Np}$ for some prime $p \equiv -b \mod q$ and the right-hand sides of equations \eqref{odd_residues} and \eqref{even_residues_e_1} vanish. 
Thus, applying equation \eqref{odd_residues} with $s_0 = 0$ and $t_0 = 3$, we obtain that $\Res_{s = 0}s\Lambda_g\left(s,\frac{b}{q},\cos\right) = 0$, and, applying equation \eqref{even_residues_e_1} with $s_0 = 0$ and $t_0 = 2$, we infer that $\Res_{s = 0} s \Lambda_g\left(s,\frac{b}{q},\sin\right)  = 0$. 

\subsubsection*{Step 2: Additive twists ($\epsilon = 0$)} 
Consider two negative rational numbers $\beta, \beta' $ with the same numerator. 
For $t_0 \in \mathbb{Z}_{>1}$ we can choose, as in the proof of Proposition~\ref{prop.Atwistsmero}, a set $T_M \subset T_\beta \cap T_{\beta'}$ of size $M$ as in Step 1, and we can find $c_\lambda \in \C$ such that equation~\eqref{lambdas} is satisfied. 
For $s_0\in\mathbb{Z}_{<1}$, applying equation~\eqref{odd_residues} to both $\beta$ and $\beta'$, then subtracting, we compute:
\begin{multline}\label{odd_residues_e_0}
	i^{-[t_0]}N^{s_0-\frac{1}{2}} \frac{(2 \pi i)^{t_0}\Gamma_{\mathbb{R}}(1-s_0+t_0)^2}{t_0!\Gamma_{\mathbb{R}}\left(1-s_0+ [ t_0 ] \right)^2} \left( \alpha^{2s_0-t_0-1} \Res_{s=s_0} (s-s_0) \Lambda_g \left( s,\beta,\cos^{([ t_0])} \right)\right.\\
	\left.   - \alpha'^{2s_0-t_0-1} \Res_{s=s_0} (s-s_0) \Lambda_g \left( s,\beta',\cos^{([ t_0])} \right)  \right) \\
	=  \sum_{\lambda \in T_M} c_\lambda \lambda^{2s_0-2t_0-1} \Res_{s = s_0}\left[ (s-s_0) \left( \Lambda_f \left( s-t_0,\alpha \lambda^{-1},\cos \right) - \Lambda_f \left( s-t_0,\alpha' \lambda^{-1},\cos \right) \right) \right],	
\end{multline}
where we write $\alpha'=-1/N\beta'$. 

Consider $\beta = \frac{b}{Nq}$ and $\beta' = \frac{b}{Nq'}$ for some $b<0$ and $q'\in\mathcal{P}$ such that $q\neq q'$ and $(b,Nq)=(b,Nq')=1$. 
Since $s_0-t_0<-\sigma$, equation~\eqref{functional_equations_additive_twists} implies that the difference $\Lambda_f(s-t_0,\alpha \lambda^{-1},\cos)-\Lambda_f(s-t_0,\alpha' \lambda^{-1},\cos)$ is holomorphic at $s = s_0$.
Thus the last line of equation \eqref{odd_residues_e_0} vanishes, so that:
\begin{equation*}
	\alpha^{2s_0-t_0-1} \Res_{s=s_0} (s-s_0) \Lambda_g \left( s,\beta,\cos^{([ t_0])} \right)   = \alpha'^{2s_0-t_0-1} \Res_{s=s_0} (s-s_0) \Lambda_g \left( s,\beta',\cos^{([ t_0])} \right)  .
\end{equation*}
Replacing $t_0$ by another integer of the same parity does not alter the residues on either side of the above equation. Choose $t_0$ of different parity to $s_0$, so that $[t_0]=[s_0+1]$. Since $\alpha \neq \alpha'$, by varying $t_0$ whilst preserving its parity, we deduce that $\Res_{s=s_0} (s-s_0) \Lambda_g \left( s,\beta,\cos^{([s_0+1])} \right)   = 0$ for all $s_0<1$. 
The condition that $b<0$ can be relaxed by mimicking the argument presented in Case 1. 

On the other hand, choosing $t_0$ to have the same parity as $s_0$, equation \eqref{double_pole_residue_epsilon_0} applies to both $\beta$ and $\beta'$. Subtracting these cases of equation \eqref{double_pole_residue_epsilon_0}, and noting that 
\[\Res_{s = s_0}\left[ (s-s_0) \left( \Lambda_f \left( s-t_0,\alpha \lambda^{-1},\cos \right) - \Lambda_f \left( s-t_0,\alpha' \lambda^{-1},\cos \right) \right) \right] = 0,\]
we obtain:
\begin{multline}	\label{even_twists_double_poles_e_1}
	 \delta_0(s_0)(-1)^{j} \frac{(1/2)_j}{(j!)} \left( \alpha^{s_0-t_0-1} -\alpha'^{s_0-t_0-1} \right) \Res_{s=1}(s-1) \Lambda_f(s)  \\
	= i^{-[t_0]} N^{s_0-\frac{1}{2}} \frac{(2 \pi i)^{t_0}\Gamma_{\mathbb{R}}(1-s_0+t_0)^2}{t_0!\Gamma_{\mathbb{R}}\left(1-s_0+ [ t_0 ] \right)^2}  \Big( \alpha^{2s_0-t_0-1} \Res_{s=s_0} (s-s_0) \Lambda_g \left(s,\beta,\cos^{([t_0])} \right)  \\
	- \alpha'^{2s_0-t_0-1}\Res_{s=s_0} (s-s_0) \Lambda_g \left(s,\beta',\cos^{([t_0])} \right)  \Big), 
\end{multline}
where $j = \frac{1}{2}(t_0-s_0)$. In the case that $s_0 < 0$, the first line of equation~\eqref{even_twists_double_poles_e_1} vanishes. It follows that: 
\begin{multline*}
	q^{2s_0-t_0-1} \Res_{s=s_0} (s-s_0) \Lambda_g \left(s,\frac{b}{Nq},\cos^{([s_0])} \right)  \\
	= q'^{2s_0-t_0-1} \Res_{s=s_0} (s-s_0) \Lambda_g \left(s,\frac{b}{Nq'},\cos^{([s_0])} \right).
\end{multline*}
By varying $t_0$ we see that $\Res_{s=s_0} (s-s_0) \Lambda_g \left(s,\frac{b}{Nq},\cos^{\left([s_0]\right)} \right) = 0$ for $s_0<0$. 

Consider $\beta = \frac{b}{q}$ for $b<0$. In this case, we have $\lambda^{-1} \alpha = \frac{q}{Np}$ for $\lambda \in T_\beta$. We conclude from the previous paragraph that $\Res_{s = 0} s \Lambda_f \left( s-3,\lambda^{-1} \alpha,\cos \right) = 0$, and so equation~\eqref{odd_residues} with $s_0 = 0$, and $t_0 = 3$ implies that $\Res_{s = 0} s \Lambda_g\left(s,\frac{b}{q},\sin\right) = 0$. 

\subsubsection*{Step 3: Odd character twists}
Combining Step 1 and Step 2, we have shown that, for $\epsilon\in\{0,1\}$, we have  $\Res_{s = 0} s \Lambda_g\left(s,\frac{b}{q},\sin\right)=-\Res_{s = 0} s \Lambda_g\left(s,\frac{b}{q},\cos^{(1)}\right) = 0$.
Varying $b~\mathrm{mod}~q$ and summing, we deduce from equation~\eqref{character_twists_in_terms_of_additive_twists} that, for any odd character $\psi$ of conductor $q$, we have $\Res_{s = 0} s \Lambda_g(\psi,s) = 0$ as required. 

\subsubsection*{Step 4: Even character twists ($\epsilon=1$)}
It was shown in Step 1 that $\Res_{s = 0}s\Lambda_g\left(s,\frac{b}{q},\cos\right) = 0$. 
Varying $b~\mathrm{mod}~q$, we deduce from equation~\eqref{character_twists_in_terms_of_additive_twists} that, for any even character $\psi$ of conductor $q$, we have $\Res_{s = 0} s \Lambda_g(\psi,s) = 0$ as required. 

\subsubsection*{Step 5: Even character twists ($\epsilon=0$)}
When $s_0 = 0$ and $t_0 = 2$, the left-hand side of equation~\eqref{double_pole_residue_epsilon_0} vanishes. Equating the right-hand side of equation~\eqref{double_pole_residue_epsilon_0} with zero, and multiplying by $2\alpha^{3}$, we observe:
\begin{equation}\label{no_dependence_on_q_double} 
	N^{-\frac{1}{2}} \frac{(2 \pi i)^2\Gamma_{\mathbb{R}}(3)^2}{\Gamma_{\mathbb{R}}\left(1\right)^2} \Res_{s=0} s \Lambda_g(s,\beta,\cos) = -\Res_{s=1} (s-1) \Lambda_f(s).
\end{equation}
Equation~\eqref{no_dependence_on_q_double} implies that $\Res_{s=0} s\Lambda_g\left(s,\frac{b}{q},\cos\right)$ does not depend on $b$ coprime to $q$. 
For a primitive even character $\psi$, we may write $\Lambda_g(s,\psi)$ as a sum of $\Lambda_g\left(s,\frac{b}{q},\cos\right)$ using equation~\eqref{character_twists_in_terms_of_additive_twists}. 
Multiplying by $s$ and taking the residue at $s=0$, we get a constant multiplied by $\sum_{b\text{ mod }q}\psi(b)$ which is zero by character orthogonality. 
It follows that $\Res_{s = 0} s \Lambda_g(\psi,s) = 0$ as required.
\end{proof}

\begin{lem}\label{lem:epsilon_0_residues_simple}
Let $\alpha,\beta,r,s_0,t_0,M,c_{\lambda}$ be as in Lemma~\ref{lem:res_1}. If $\epsilon = 0$, then:
	
\begin{multline}\label{epsilon_0_residues_simple}
 \Res_{s = s_0} \left[ \sum_{\lambda \in T_{\beta,M}} c_\lambda \lambda^{2s-2t_0-1} \Lambda_f \left( s-t_0,\alpha \lambda^{-1},\cos \right) \right]\\
	=i^{-[t_0]} \left( N \alpha^2 \right)^{s_0-\frac{1}{2}} \alpha^{-t_0} \frac{(2 \pi i)^{t_0}}{t_0!} \Bigg[ \log(N\alpha^2) \frac{\Gamma_{\mathbb{R}}(1-s_0+t_0)^2}{\Gamma_{\mathbb{R}}\left(1-s_0+ [ t_0 ] \right)^2}\Res_{s=s_0}(s-s_0) \Lambda_g \left(s,\beta,\cos^{([t_0])} \right)  \\
	 + \frac{d}{ds}    \frac{\Gamma_{\mathbb{R}}(1-s+t_0)^2}{\Gamma_{\mathbb{R}}\left(1-s+ [ t_0 ] \right)^2}\Bigg|_{\substack{s=s_0}} \Res_{s=s_0}(s-s_0) \Lambda_g \left(s,\beta,\cos^{([t_0])} \right)  \\
	 + \frac{\Gamma_{\mathbb{R}}(1-s_0+t_0)^2}{\Gamma_{\mathbb{R}}\left(1-s_0+ [ t_0 ] \right)^2}\Res_{s=s_0}\Lambda_g \left(s,\beta,\cos^{([t_0])} \right)  \Bigg] \\
	 +(-1)^j \alpha^{s_0-t_0-1} \Big[\delta_0(s_0) \frac{d}{ds}\left(J_j(s)G_j(s)\right)\Big|_{\substack{s=1}} \Res_{s=1}  (s-1) \Lambda_f(s)  \\
-2 \delta_0(s_0) \frac{\left( \frac{1}{2} \right)_j}{j!} \log(\alpha) \Res_{s=1}  (s-1) \Lambda_f(s)  + \delta_0(s_0) \frac{\left( \frac{1}{2} \right)_j}{j!}  \Res_{s=1} \Lambda_f(s) \Big].
\end{multline}
\end{lem}
\begin{proof}
Proceeding as in the proof of Lemma~\ref{lem:residue_2}, we observe:
\begin{multline}\label{eq.rlasII}
		\Res_{s=s_0} \left[ (\lambda \alpha)^{s-t_0-\frac{1}{2}} \sum_{k=0}^{\ell_0-1} \left(\frac{\widetilde{\mathcal{I}}_k(\alpha \lambda^{-1})}{s-t_0+2k}+2\frac{\mathcal{I}_k(\alpha \lambda^{-1})}{(s-t_0+2k)^2} \right) \right] \\
		= (\lambda \alpha)^{s_0-t_0-\frac{1}{2}}  \widetilde{\mathcal{I}}_j(\alpha \lambda^{-1}) + 2 \log(\alpha \lambda) (\alpha \lambda)^{s_0-t_0-\frac{1}{2}} \mathcal{I}_j(\alpha \lambda^{-1}),
\end{multline}
where $j$ as in Lemma~\ref{lem:res_1}. We may evaluate $\mathcal{I}_j(\alpha\lambda^{-1})$ as in equation~\eqref{eq.I_twiddle_alpha_lambda}. On the other hand, recalling that $\Lambda_f(s)$ is holomorphic away from at most double poles in the set $\{0,1\}$, we compute:
\begin{equation*}
\widetilde{\mathcal{I}}_j(\alpha \lambda^{-1}) = (-1)^j \frac{\sqrt{\pi}}{(j!)^2} \sum_{p \in \{0,1 \} } \Res_{s = p} \left[ \Lambda_f(s) J_j(s)G_j(s) (\alpha \lambda^{-1})^{\frac{1}{2}-s}\right],
\end{equation*}
where $G_j(s)$ is as in equation \eqref{eq:G} and $J_j(s)$ is as in equation \eqref{eq:K}, with $\epsilon=0$. Since $J_j(s)G_j(s)$ has a double zero at $s=0$, but is non-zero at $s=1$, we deduce
\begin{multline} \label{eq.residue_I_twiddle}
	\widetilde{\mathcal{I}}_j(\alpha \lambda^{-1}) = (-1)^j (\alpha \lambda^{-1})^{-\frac{1}{2}} \frac{\sqrt{\pi}}{(j!)^2} \Bigg( \frac{d}{ds}\left(J_j(s)G_j(s)\right)\Big|_{\substack{s=1}} \Res_{s=1}(s-1)\Lambda_f(s)\\
	 -  J_j(1)G_j(1)  \log(\alpha \lambda^{-1}) \Res_{s=1}(s-1)\Lambda_f(s) + J_j(1)G_j(1)  \Res_{s=1}  \Lambda_f(s)\Bigg).
\end{multline}
Observing the evaluation $\sqrt{\pi}~J_j(1)G_j(1) =\left(1/2\right)_j j!$, and using equations~\eqref{eq.I_twiddle_alpha_lambda}, ~\eqref{eq.rlasII}, and~\eqref{eq.residue_I_twiddle}, we obtain
\begin{multline*}
\Res_{s=s_0} \left[ \sum_{\lambda \in T_{\beta,M}}c_\lambda  (\lambda \alpha)^{s-t_0-\frac{1}{2}} \sum_{k=0}^{\ell_0-1} \left(\frac{\widetilde{\mathcal{I}}_k(\alpha \lambda^{-1})}{s-t_0+2k}+2\frac{\mathcal{I}_k(\alpha \lambda^{-1})}{(s-t_0+2k)^2} \right) \right] \\
= (-1)^j \alpha^{s_0-t_0-1} \Bigg[\delta_0(s_0) \frac{\sqrt{\pi}}{(j!)^2}\frac{d}{ds}\left(J_j(s)G_j(s)\right)\Big|_{\substack{s=1}} \Res_{s=1} (s-1) \Lambda_f(s)  \\
-2 \delta_0(s_0) \frac{\left( 1/2\right)_j}{j!} \log(\alpha) \Res_{s=1} (s-1) \Lambda_f(s)  + \delta_0(s_0) \frac{\left( 1/2 \right)_j}{j!}  \Res_{s=1} \Lambda_f(s) \Bigg],
\end{multline*}
in which we note that the $\log(\lambda)$ terms cancel. 
On the other hand, we compute directly:
\begin{multline*}
	\Res_{s = s_0}\left[ i^{-[t_0]} \left( N \alpha^2 \right)^{s-\frac{1}{2}} \alpha^{-t_0} \frac{ (2 \pi i)^{t_0}\Gamma_{\mathbb{R}}(1-s+t_0)^2}{t_0!\Gamma_{\mathbb{R}}\left(1-s+ [ t_0 ] \right)^2}\Lambda_g \left( s,\beta,\cos^{([t_0])} \right) \right] \\
	= i^{-[t_0]} \left( N \alpha^2 \right)^{s_0-\frac{1}{2}} \alpha^{-t_0} \frac{(2 \pi i)^{t_0}}{t_0!} \Bigg[ \log(N\alpha^2)\frac{\Gamma_{\mathbb{R}}(1-s_0+t_0)^2}{\Gamma_{\mathbb{R}}\left(1-s_0+ [ t_0 ] \right)^2} \Res_{s=s_0}(s-s_0) \Lambda_g \left(s,\beta,\cos^{([t_0])} \right) \\
	 + \frac{d}{ds}    \frac{\Gamma_{\mathbb{R}}(1-s+t_0)^2}{\Gamma_{\mathbb{R}}\left(1-s+ [ t_0 ] \right)^2}\Bigg|_{\substack{s=s_0}} \Res_{s=s_0}(s-s_0) \Lambda_g \left(s,\beta,\cos^{([t_0])} \right) \\
	+ \frac{\Gamma_{\mathbb{R}}(1-s_0+t_0)^2}{\Gamma_{\mathbb{R}}\left(1-s_0+ [ t_0 ] \right)^2}\Res_{s=s_0}\ \Lambda_g \left(s,\beta,\cos^{([t_0])} \right) \Bigg].
\end{multline*}
By assumption $s_0>t_0-2\ell_0 + \frac{3}{2}$, and so the function in equation~\eqref{long_lambda} is holomorphic at $s=s_0$. Equation~\eqref{epsilon_0_residues_simple} follows by taking the residue of this function at $s=s_0$.
\end{proof}
\begin{prop}
Make the assumptions of Theorem \ref{thm:0Converse}. If $\psi$ is a primitive Dirichlet character with conductor $q \in \mathcal{P}$, then $\Lambda_f(s,\psi)$ and $\Lambda_g(s,\psi)$ are entire and bounded in vertical strips.	
\end{prop}

That is, assumption (3) in Theorem~\ref{thm:EntireConverse} is valid.

\begin{proof}
Since $\Lambda_f(s,\psi)$ and $\Lambda_g(s,\psi)$ are in $\mathcal{M}(-\infty,\infty)$, it suffices to prove entirety. 
Mimicking the argument surrounding equation~\eqref{eq.suffices}, we see it suffices to show that $\Res_{s=0}\Lambda_g(s,\psi)=0$ (reversing the roles of $f$ and $g$ will yield that $\Res_{s=0}\Lambda_f(s,\psi)=0$).

Given $s_0 \in \mathbb{Z}_{<1}$, choose $\ell_0,M,t_0\in\mathbb{Z}_{>1}$ such that $M\geq2\ell_0>t_0-s_0+\frac32$. 
For $\alpha\in\mathbb{Q}_{>0}$ and $\beta=-1/N\alpha$ choose $T_{\beta,M}\subset T_{\beta}$ of size $M$, and, for each $\lambda\in T_{\beta,M}$, choose $c_{\lambda}\in\mathbb{C}$ satisfying equation~\eqref{lambdas}. 
As in the proof of Lemma \ref{lem:twistssimplepoles}, function in equation \eqref{long_lambda} is holomorphic at $s=s_0$. 
Taking the residue at $s = s_0$, we get:
\begin{multline} \label{odd_residues_simple}
	\Res_{s=s_0} \left[i^{-[\epsilon + t_0]}(N \alpha^2)^{s-\frac{1}{2}}\alpha^{-t_0} \frac{(2 \pi i)^{t_0}\Gamma_{\mathbb{R}}(1-s+t_0)^2}{t_0!\Gamma_{\mathbb{R}}\left(1-s+ [ t_0 ] \right)^2}\Lambda_g \left( s,\beta,\cos^{([\epsilon + t_0])} \right)  \right] \\
	=  \Res_{s = s_0}\left[\sum_{\lambda \in T_{\beta,M}} c_\lambda \lambda^{2s-2t_0-1} (-i\pi)^\epsilon   \Lambda_f \left( s-t_0,\alpha \lambda^{-1},\cos^{(\epsilon)} \right) \right].
\end{multline}
Following the strategy from Lemma \ref{lem:twistssimplepoles}, we split the remainder of the proof into five steps.
\subsubsection*{Step 1: Additive twists ($\epsilon = 1$)}  
Consider first $\beta = \frac{b}{Nq}$ for some $b<0$ coprime to $Nq$. 
Using our assumptions on $s_0$, $t_0$, $\{a_n\}_{n=1}^{\infty}$, and $\sigma$, we deduce from equation~\eqref{functional_equations_additive_twists} that $\Lambda_f(s-t_0,\alpha \lambda^{-1},\sin)$ is holomorphic at $s = s_0$. 
In particular, the right-hand side of equation \eqref{odd_residues_simple} is zero. Choose $t_0$ of different parity to $s_0$, so that $[\epsilon+t_0]=[s_0]$.  In the proof of Lemma \ref{lem:character_only_simple}, it was demonstrated that 
$\Lambda_g \left( s,\beta,\cos^{([s_0])} \right)$ has at most a simple pole at $s=s_0$. Therefore, the left hand side of equation~~\eqref{odd_residues_simple} is equal to 
\[
i^{-[\epsilon + t_0]}(N \alpha^2)^{s_0-\frac{1}{2}}\alpha^{-t_0} \frac{(2 \pi i)^{t_0}\Gamma_{\mathbb{R}}(1-s_0+t_0)^2}{t_0!\Gamma_{\mathbb{R}}\left(1-s_0+ [ t_0 ] \right)^2} \Res_{s=s_0} \Lambda_g \left( s,\beta,\cos^{([s_0])} \right) .
\]
It follows that $\Res_{s=s_0}\Lambda_g \left( s,\beta,\cos^{([s_0)]} \right) = 0$ for all $s_0<1$. 

In the proof of Lemma \ref{lem:character_only_simple} it was established that the function $\Lambda_g \left(s,\beta,\cos^{([1+s_0])} \right)$ has at most a simple pole at $s = s_0$. Taking $t_0$ of the same parity to $s_0$, we may apply equation \eqref{messy_1} to obtain:
\begin{multline}\label{even_residues_e_1_simple}
		i^{-[1+t_0]} (N\alpha^2)^{s_0-1/2} \alpha^{-t_0} \frac{(2 \pi i)^{t_0}\Gamma_{\mathbb{R}}(1-s+t_0)^2}{t_0!\Gamma_{\mathbb{R}}\left(1-s+ [ t_0 ] \right)^2} \Res_{s=s_0} \Lambda_g \left(s,\beta,\cos^{([1+s_0])} \right)  \\
		=\Res_{s=s_0} \left[ i \pi \sum_{\lambda \in T_{\beta,M}} c_\lambda \lambda^{2s-2t_0-1}  \Lambda_f \left(s-t_0,\alpha \lambda^{-1},\sin \right)  \right].
\end{multline}		
The second line of equation \eqref{even_residues_e_1_simple} vanishes, and so $\Res_{s=s_0} \Lambda_g \left( s,\frac{b}{Nq},\cos^{([1+s_0])} \right) = 0$ for all $s_0 <1$. 

If, instead, $\beta = \frac{b}{Nq}$ for some $b>0$ such that $(b,Nq)=1$, then we may reduce to the case $b<0$ as in the proof of Lemma \ref{lem:character_only_simple}. 

Consider $\beta = \frac{b}{q}$ for $b<0$ coprime to $q$, so that $\lambda^{-1} \alpha = \frac{q}{Np}$ for a prime $p \equiv -b \mod q$. As above, the second lines of equations \eqref{odd_residues_simple} and \eqref{even_residues_e_1_simple} vanish. Applying equation~\eqref{odd_residues_simple} with $s_0 = 0$, $t_0=3$ gives $\Res_{s=0} \Lambda_g\left(s,\frac{b}{q},\cos\right)  = 0$, and applying equation~\eqref{even_residues_e_1_simple} with $s_0 = 0$, $t_0=2$ gives $\Res_{s=0} \Lambda_g\left(s,\frac{b}{q},\sin\right)= 0$. 

\subsubsection*{Step 2: Additive twists ($\epsilon = 0$)}  
Consider $\beta=\frac{b}{Nq}$ with $b<0$, and let $\beta' \in \mathbb{Q}_{<0}$ have the same numerator. For $t_0 \in \mathbb{Z}_{\geq0}$ we can choose as in the proof of Proposition~\ref{prop.Atwistsmero} a set $T_M \subset T_\beta \cap T_{\beta'}$ with $M$ as in Step 1, and we can find $c_\lambda \in \C$ such that equation~\eqref{lambdas} is satisfied. Taking $t_0$ of different parity to $s_0$, equation~\eqref{odd_residues_simple} applies to both $\beta$ and $\beta'$.  Subtracting these cases of equation~\eqref{odd_residues_simple}, we get
\begin{multline}\label{odd_residues_e_0_simple}
	i^{-[t_0]}N^{s_0-\frac{1}{2}} \frac{(2 \pi i)^{t_0}\Gamma_{\mathbb{R}}(1-s+t_0)^2}{t_0!\Gamma_{\mathbb{R}}\left(1-s+ [ t_0 ] \right)^2}\\
	\cdot \left( \alpha^{2s_0-t_0-1} \Res_{s=s_0} \Lambda_g \left( s,\beta,\cos^{([1+s_0])} \right)  - \alpha'^{2s_0-t_0-1} \Res_{s=s_0}  \Lambda_g \left( s,\beta',\cos^{([1+s_0])} \right) \right) \\
	=  \Res_{s=s_0} \left[ \sum_{\lambda \in T_M} c_\lambda \lambda^{2s-2t_0-1} \left( \Lambda_f \left( s-t_0,\alpha'\lambda^{-1},\cos \right) - \Lambda_f \left( s-t_0,\alpha \lambda^{-1},\cos \right) \right) \right],	
\end{multline}
where we have written $\alpha'=-1/N\beta'$ and used the fact that the functions $\Lambda_g \left(s,\beta,\cos^{([1+s_0])}\right)$ and $\Lambda_g \left(s,\beta',\cos^{([1+s_0])}\right)$ have at most a simple pole at $s = s_0$ (which was established in the proof of Lemma \ref{lem:character_only_simple}). 
As per the argument following equation~\eqref{odd_residues_e_0}, the last line of equation~\eqref{odd_residues_e_0_simple} vanishes. 
It follows that:
\begin{equation*}
	\alpha^{2s_0-t_0-1} \Res_{s=s_0}  \Lambda_g \left( s,\beta,\cos^{([1+s_0])} \right)   = \alpha'^{2s_0-t_0-1} \Res_{s=s_0}   \Lambda_g \left( s,\beta',\cos^{([1+s_0])} \right)  .
\end{equation*}
Varying $t_0$, whilst keeping the same parity, we deduce that $\Res_{s=s_0}  \Lambda_g \left(s,\frac{b}{Nq},\cos^{([1+s_0])} \right) = 0$, for all $s_0<1$ and $b<0$ coprime to $q$. 

Taking $t_0$ of the same parity as $s_0$, equation \eqref{epsilon_0_residues_simple} applies to both $\beta$ and $\beta'$. Arguing as above, we are lead to:
 \begin{equation*}
 	q^{2s_0-t_0-1} \Res_{s = s_0} \Lambda_g\left(s,\frac{b}{Nq},\cos^{([s_0])} \right) = q'^{2s_0-t_0-1} \Res_{s = s_0} \Lambda_g\left(s,\frac{b}{Nq'},\cos^{([s_0])} \right).
 \end{equation*}
Varying $t_0$, we conclude that $\Res_{s = s_0} \Lambda_g\left(s,\frac{b}{Nq},\cos^{([s_0])} \right) = 0$ for all $s_0 <0$. If, instead, $\beta = \frac{b}{Nq}$ for some $b>0$ such that $(b,Nq)=1$, then we may reduce to the case $b<0$ as in the proof of Lemma \ref{lem:character_only_simple}. 

Consider $\beta = \frac{b}{q}$ with $b<0$, so that $\lambda^{-1} \alpha = \frac{q}{Np}$ for $\lambda \in T_\beta$. As per the previous paragraph, we find that $\Res_{s = 0}  \Lambda_f \left( s-3,\lambda^{-1} \alpha,\cos \right)  = 0$. 
From Lemma \ref{lem:twistssimplepoles} we know $\Res_{s = 0} s \Lambda_g\left(s,\frac{b}{q},\sin\right) = 0 $.  
Applying equation~\eqref{odd_residues_simple} with $s_0=0$ and $t_0 = 3$, we therefore see that $\Res_{s = 0} \Lambda_g\left(s,\frac{b}{q},\sin\right) = 0$. 

\subsubsection*{Step 3: Odd character twists}
Combining Step 1 and Step 2, we have shown that, for $\epsilon\in\{0,1\}$, we have  $\Res_{s = 0} \Lambda_g\left(s,\frac{b}{q},\sin\right)=-\Res_{s = 0} \Lambda_g\left(s,\frac{b}{q},\cos^{(1)}\right) = 0$.
Varying $b~\mathrm{mod}~q$ and summing, we deduce from equation~\eqref{character_twists_in_terms_of_additive_twists} that, for any odd character $\psi$ of conductor $q$, we have $\Res_{s = 0} \Lambda_g(\psi,s) = 0$ as required. 

\subsubsection*{Step 4: Even character twists ($\epsilon=1$)}
It was shown in Step 1 that $\Res_{s = 0}\Lambda_g\left(s,\frac{b}{q},\cos\right) = 0$. 
Varying $b~\mathrm{mod}~q$, we deduce from equation~\eqref{character_twists_in_terms_of_additive_twists} that, for any even character $\psi$ of conductor $q$, we have $\Res_{s = 0} \Lambda_g(\psi,s) = 0$ as required. 

\subsubsection*{Step 5: Even character twists ($\epsilon=0$)}
When $s_0=0$ and $t_0=2$, the left-hand side of equation~\eqref{epsilon_0_residues_simple} vanishes. 
Substituting in equation \eqref{no_dependence_on_q_double} and dividing through by $\alpha^{-3}$, we see that:
\begin{multline}\label{ddob}
	\frac{\Gamma_{\R}(3)^2}{\Gamma_{\R}(1)^2} \Res_{s=0}\Lambda_g(s,\beta,\cos)  \\
= - \left(\log(N)\frac{\Gamma_{\R}(3)^2}{\Gamma_{\R}(1)^2} + \frac{d}{ds}   \frac{\Gamma_{\R}(3-s)^2}{\Gamma_{\R}(1-s)^2} \Bigg|_{\substack{s=0}} \right)\Res_{s=0}s \Lambda_g(s,\beta,\cos)   \\
	+ N^{\frac{1}{2}} \frac{2}{(2 \pi i)^2}\frac{d}{ds}\left(J_j(s)G_j(s)\right)\Big|_{s=1} \Res_{s=1} (s-1) \Lambda_f(s).
\end{multline}
Since we make the same assumptions as in Lemma~\ref{lem:twistssimplepoles}, equation~\eqref{no_dependence_on_q_double} holds.
Substituting equation~\eqref{ddob} into equation~\eqref{no_dependence_on_q_double}, we deduce that $\Res_{s=0}s\Lambda_g(s,\beta,\cos)$ does not depend on $b$ coprime to $q$. 
Therefore $\Res_{s=0}\Lambda_g(s,\beta,\cos)$ is also independent of $b$ coprime to $q$. 
As in the proof of Lemma~\ref{lem:twistssimplepoles}, it follows from equation~\eqref{character_twists_in_terms_of_additive_twists} and character orthogonality that, for even primitive characters $\psi$, we have $\Res_{s=0} \Lambda_g(s,\psi)=0$.

\end{proof}

Altogether, we have established that the assumptions of Theorem~\ref{thm:0Converse} imply those of Theorem~\ref{thm:EntireConverse}.

\section{Proof of Theorem~\ref{thm.ArtinQuotients}}\label{sec:mainproof}
We will use the notation of Section~\ref{sec.intro}.
From equation~\eqref{eq.CancelGamma}, we recall the quotient $L(s)=L(s,\phi)/\zeta(s)$, and, denoting by $\tilde{\phi}$ the contragredient to $\phi$, we introduce analogously $\tilde{L}(s)=L(s,\tilde{\phi})/\zeta(s)$. Both $L(s)$ and $\tilde{L}(s)$ may be written as Dirichlet series:
\[
L(s)=\sum_{n=1}^{\infty}a_nn^{-s}, \ \ \tilde{L}(s)=\sum_{n=1}^{\infty}c_nn^{-s}.
\]
One can compute the Dirichlet coefficients $a_n$ (resp. $c_n$) of $L(s)$ (resp. $\tilde{L}(s)$) by first expanding each factor in equation~\eqref{ArtinEP} (resp. the analogous equation for $\tilde{\phi}$) as a geometric series, and subsequently performing the necessary Dirichlet series manipulations.
Doing so, we deduce that $a_n=O(n^{\sigma})$ (resp. $c_n=O(n^{\sigma})$) for all $\sigma>0$.

For any primitive Dirichlet character $\psi$ modulo $q$, the Dirichlet $L$-function $L(s,\psi)$ has completion $\Lambda(s,\psi)=\Gamma_{\mathbb{R}}(s+\epsilon_{\psi})L(s,\psi)$, where $\epsilon_{\psi}\in\{0,1\}$ is such that $\psi(-1)=(-1)^{\epsilon_{\psi}}$. 
As reviewed in \cite[Chapter~1]{ILP}, the completion $\Lambda(s,\psi)$ satisfies the functional equation:
\begin{equation}\label{eq.FED}
\Lambda(s,\psi)=i^{-\epsilon_{\psi}}\frac{\tau(\psi)}{q^{1/2}}q^{\frac12-s}\Lambda(1-s,\bar{\psi}).
\end{equation}
Recall that $p$ (resp. $m$) denotes the dimension of the $(+1)$-eigenspace (resp. $(-1)$-eigenspace) of $\phi(c)$, and consider first the case that $p\in\{1,3\}$. 
Let $N$ denote the conductor of $\phi$, and let $\mathcal{P}$ denote the set of odd primes coprime to $N$.
For any primitive $\psi$ mod $q\in\mathcal{P}$, the tensor product representation $\phi\otimes\psi$ (resp. $\tilde{\phi}\otimes\bar{\psi}$) is a $3$-dimensional Artin representation of $\mathrm{Gal}(\overline{\mathbb{Q}}/\mathbb{Q})$, with completed $L$-function $\Lambda(s,\phi\otimes\psi)=\Gamma_{\mathbb{R}}(s+\epsilon_{\psi})^p\Gamma_{\mathbb{R}}(s+[\epsilon_{\psi}+1])^mL(s,\phi\otimes\psi)$. 
Similarly to equation~\eqref{eq.CancelGamma}, we have
\begin{equation}\label{eq.twistquotient}
\begin{split}
\frac{\Lambda(s,\phi\otimes\psi)}{\Lambda(s,\psi)}=\Gamma_{\mathbb{R}}(s+\epsilon_{\psi})^{p-1}\Gamma_{\mathbb{R}}(s+[\epsilon_{\psi}+1])^m\sum_{n=1}^{\infty}a_n\psi(n)n^{-s}, \\
\left(\text{resp.} \ \ \frac{\Lambda(s,\tilde{\phi}\otimes\bar{\psi})}{\Lambda(s,\psi)}=\Gamma_{\mathbb{R}}(s+\epsilon_{\psi})^{p-1}\Gamma_{\mathbb{R}}(s+[\epsilon_{\psi}+1])^m\sum_{n=1}^{\infty}c_n\bar{\psi}(n)n^{-s}\right).
\end{split}
\end{equation}
As reviewed in \cite[Chapter~4]{ILP}, the completed Artin $L$-function $\Lambda(s,\phi\otimes\psi)$ satisfies a functional equation of the form: 
\begin{equation}\label{eq.FEAshape}
\Lambda(s,\phi\otimes\psi)=\epsilon(s,\phi\otimes\psi)\Lambda(1-s,\tilde{\phi}\otimes\bar{\psi}).
\end{equation}
Appling the formulae in \cite[Section~5]{Deligne}, we get
\begin{equation}\label{eq.wppwp}
\epsilon(s,\phi\otimes\psi)=w(\phi)i^{-3\epsilon_{\psi}}\chi(q)\psi(N)\frac{\tau(\psi)^3}{q^{3/2}}\left(Nq^3\right)^{\frac12-s},
\end{equation}
for some character $\chi$ mod $N$ (not necessarily primitive),
for some $w(\phi)=w(\phi\otimes\textbf{1})$ with absolute value $1$.
According to \cite[Section~3.11]{Deligne}, we may write $w(\phi)$ as a product of local factors $w(\phi)=\prod_vw_v(\phi_v)$, in which the index $v$ varies over the places of $F$. 
Denoting by $w_{\infty}(\phi)$ (resp. $w_{<\infty}(\phi)$) the product over the archimedean (resp. non-archimedean) places, we have $w(\phi)=w_{\infty}(\phi)w_{<\infty}(\phi)$.
Using \cite[equations~(3.6),~(3.7)]{Knapp}, at each archimedean place we may infer the local epsilon factor from the corresponding gamma factor written in canonical form. Doing so, we ultimately calculate $w_{\infty}(\phi)=i^m$. 

Combining this with equations \eqref{eq.FEAshape} and \eqref{eq.wppwp}, we deduce that $\Lambda(s,\phi\otimes\psi)$ satisfies the functional equation:
\begin{equation}\label{eq.FEA}
\Lambda(s,\phi\otimes\psi)=w_{<\infty}(\phi)i^{m-3\epsilon_{\psi}}\chi(q)\psi(N)\left(Nq^{3}\right)^{1/2-s}\Lambda(1-s,\tilde{\phi}\otimes\bar{\psi}).
\end{equation}
For each primitive $\psi$ mod $q\in\mathcal{P}$, dividing equation \eqref{eq.FEA} by equation~\eqref{eq.FED}, using equation~\eqref{eq.twistquotient}, and noting that $i^{m-2\epsilon_{\psi}}=(-1)^{m/2-\epsilon_{\psi}}$, we recover equation~\eqref{eq:FE} with $b_n=w_{<\infty}(\phi)c_n$ and $\epsilon=m/2\in\{0,1\}$.
Assuming for a contradiction that $\Lambda(s,\phi)/\xi(s)$ has finitely many poles, Theorem~\ref{thm:0Converse} implies that $L(s)=L_f(s)$ for some weight 0 Maass form $f$. 
Therefore, we may write:
\[
f(z)=\sum_k \alpha_kh_k(z) + \sum_{\ell}\beta_{\ell}E_{\ell}(z),
\]
for cuspidal Hecke eigenforms $h_k$ and Eisenstein series $E_{\ell}$. 
It follows that:
\begin{equation}\label{eq.LHE}
L_f(s)=\sum_k\alpha_kL_{h_k}(s)+\sum_{\ell}\beta_{\ell}L_{E_{\ell}}(s).
\end{equation}
By construction $L_f(s)$ is equal to the quotient $L(s,\phi)/\zeta(s)$, and so:
\begin{equation}\label{eq.lineardependence}
L(s,\phi)=\zeta(s)L_f(s)=\zeta(s)\left(\sum_{k}\alpha_kL_{h_k}(s)+\sum_{\ell}\beta_{\ell}L_{E_{\ell}}(s)\right).
\end{equation}
Since we assumed that $L(s,\phi)$ was primitive, and $\zeta(s)$ has an Euler product, it cannot happen that $L_f(s)$ has an Euler product.
Since each $L_{h_k}(s)$ (resp. $L_{E_{\ell}}(s)$) does have an Euler product, it follows that $\#\{k\}+\#\{\ell\}>1$. 
Equation~\eqref{eq.lineardependence} therefore contradicts the main theorem \cite{KMP}, which establishes the linear independence of Euler products in a large axiomatic class including Artin $L$-functions and automorphic $L$-functions.

The case  $p=2$ is similar. Indeed, in this case, the gamma factor in equation~\eqref{eq.CancelGamma} is equal to (half of) the gamma factor $\Gamma_{\C}(s)=2\Gamma_{\R}(s)\Gamma_{\R}(s+1)$ for a holomorphic modular form. Dividing equation~\eqref{eq.FEA} by equation~\eqref{eq.FED}, we get \cite[equation~(1)]{WCTWP}. 
If the quotient $\Lambda(s,\phi)/\xi(s)$ had only finitely many poles, then \cite[Theorem~1.1]{WCTWP} would imply that it is the completed $L$-function of a weight 1 modular form and we may argue as above.

\end{document}